\documentclass[11pt, reqno]{amsart}

\usepackage{amsmath,amssymb,latexsym}
\usepackage{color}
\usepackage{xcolor}
\usepackage{graphicx}
\usepackage{cite}
\usepackage[colorlinks=true]{hyperref}
\usepackage{marvosym}
\hypersetup{urlcolor=blue, citecolor=red, linkcolor=blue}
\usepackage{ulem}

\newcommand{\R}{\mathbb{R}}

\newcommand{\diff}{\mathrm{d}}

\newtheorem{theorem}{Theorem}[section]
\newtheorem{lemma}{Lemma}
\newtheorem{proposition}{Proposition}
\newtheorem{remark}{Remark}
\newtheorem{definition}{Definition}

\newtheorem*{main-theorem}{Main Theorem}
\newtheorem*{remark*}{Remark}
\newtheorem*{lemma*}{Lemma A.1}

\numberwithin{equation}{section}

\begin{document}
	
	\title[The generalized Fornberg-Whitham equation]{Wave breaking for the generalized Fornberg-Whitham equation}

	\author{Jean-Claude Saut, Shihan Sun, Yuexun Wang, and Yi Zhang}
	\address{Université Paris-Saclay, CNRS, Laboratoire de Mathématiques d'Orsay, 91405 Orsay, France}
	\email{jean-claude.saut@universite-paris-saclay.fr}

	\address{School of Mathematics and Statistics, Lanzhou University, 730000 Lanzhou,
		People's Republic of China}
	
\email{sunshh21@lzu.edu.cn}

\email{yuexunwang@lzu.edu.cn}

\email{yizhang@lzu.edu.cn}
	
	
	\begin{abstract}
		This paper aims to show that the Cauchy problem of the Burgers equation with a weakly dispersive perturbation involving  the Bessel potential (generalization of the Fornberg-Whitham equation) can exhibit wave breaking for initial data with large slope. We also comment on the dispersive properties of the equation.
	\end{abstract}
	\maketitle

	\section{introduction}
	
	This paper is a continuation of a previous work, \cite{MR4409228}, aiming to understand the possible wave breaking in weak perturbations of the Burgers equation. This kind of equations  is a toy model to understand the influence of a weakly dispersive perturbations of scalar or system conservation laws. Actually in most physically relevant dispersive systems  ({\it eg} the water wave system, see \cite{La}), the dispersion is weak and strong dispersive effects occur for instance in a long wave limit after Taylor expanding the dispersion relation. On the other hand the nonlinearity is often quadratic (coming for particular from the Euler equation). It thus appears that equations with high order nonlinear terms and high dispersion such as the generalized Korteweg-de Vries equation are not appropriate to analyze the problem under study.
	
	The possibility of wave breaking for the fractionary Korteweg-de Vries equation (fKdV)
	\begin{equation*}\label{fKdV}
\begin{aligned}
		\partial_tu+u\partial _xu-(-\partial_x^2)^{\alpha/2}\partial_x u=0,\quad -1\leq \alpha<0,
\end{aligned}
	\end{equation*}
and of the related Whitham equation, \cite{MR0671107, MR1699025}, has been proven in many recent papers \cite{MR4409228, MR3682673, 2021arXiv210707172O, MR4321245}. 
\footnote{The earlier works \cite{MR1261868, MR1668586} have shown wave
breaking for a wide class of non-local dispersive equations, however, which does not apply directly to the Whitham equation. The blow-up results in \cite{MR2727172} concern the blow-up of the $C^{1+\delta}$ norm,  but the boundedness of the solution is not proven.}
We aim here to show similar results in the case where the dispersive operator has a smooth symbol.\footnote{The Whitham equation has also a smooth symbol.}
	
	More precisely, we are concerned with the Cauchy problem of non-local weakly dispersive perturbations of the Burgers equation (will be referred to as the generalized Fornberg-Whitham equation)
	\begin{equation}\label{eq:1.1}
		\bigg\{\begin{aligned}
			&\partial_{t} u + u^p \partial_{x} u - \mathcal{K}_s \partial_{x} u = 0,\quad p=1,2,...\\
			&u(x,0)=u_{0}(x),
		\end{aligned}\bigg.
	\end{equation}
	where $(x, t)\in \mathbb{R}\times \mathbb{R}$, and  $\mathcal{K}_s = (1- \partial_{x}^{2})^{- \frac{s}{2}}(=G_{s}*)$ is the Bessel potential of order $s>0$, whose Fourier multiplier is given by $(1+|\xi|^2)^{-\frac{s}{2}}$. Furthermore,
	the kernel \(G_{s}\) has the formula (see for instance \cite{MR1411441, MR0374877})
	\begin{equation*}
\begin{aligned}
		G_{s}(x) = c_0(s)\int_{0}^{\infty} t^{\frac{s-3}{2}} e^{-\frac{|x|^2}{4t}-t}\, \mathrm{d} t. 
\end{aligned}
	\end{equation*}
	
	The advantage of considering a dispersive perturbation with Bessel potential instead for instance of the dispersion $(-\partial_x^2)^{\alpha/2}\ (-1<\alpha<0)$ of the aforementioned fKdV equation,  is that one always gets a smooth phase velocity and that the dispersion can have arbitrarily small order. Another advantage is that in the long wave limit one obtains formally the Korteweg-de Vries equation and thus one keeps a dispersive regime, as it is the case for the Whitham equation, see below.
	
	Notice that when $s=1/2$ the dispersion is reminiscent of that of the linearized gravity water waves system in infinite depth. 
	
	 Also the case $s=2, p=1$  corresponds to  the non-local version of the Fornberg-Whitham equation originally written as \cite{MR497916}:	
	\begin{equation}\label{FW}
\begin{aligned}
	u_t-u_{xxt}-3u_xu_{xx}-uu_{xxx}+uu_x+u_x=0,
\end{aligned}
	\end{equation}
which is obtained after applying the operator $(I-\partial_x^2)^{-1/2}$ to \eqref{FW}, namely: 
		\begin{equation}\label{FWbis}
	u_t+uu_x+(I-\partial_x^2)^{-1}u_x=0.
	\end{equation}
Note that \eqref{FWbis} is also equivalent to
	\begin{equation}\label{FWter}
\begin{aligned}
	u_t+uu_x+K\star u_x=0,
\end{aligned}
	\end{equation}
where the convolution kernel $K$ is given by $K(x)=\frac{e^{-|x|}}{2}$. 
	
	On the form \eqref{FWbis}, the Fornberg-Whitham equation \eqref{FW} is sometimes known as the Burgers-Poisson equation, see for instance \cite{MR2084199, MR2164195}, and can be written in the system form: 
	\begin{equation}\label{BPsys}
		\bigg\{\begin{aligned}
			&u_t + u u_x -   \phi_x = 0,\\
			&\phi_{xx}=\phi+u.
		\end{aligned}\bigg.
	\end{equation}
	
The Fornberg-Whitham equation \eqref{FW} was introduced by Whitham \cite{MR0671107}, see also \cite{MR1699025}. A first systematic numerical study was given in \cite{MR497916}. We refer to \cite{MR4270781} for a nice review of known results on the Fornberg-Whitham equation.

	\begin{remark}
	On the form of equation \eqref{FW}, the Fornberg-Whitham equation is reminiscent of a well-known asymptotic model of shallow waves, namely the Camassa-Holm equation (see \cite{MR1234453,MR636470}):
	\begin{equation}\label{MR1234453}
\begin{aligned}
	u_t-u_{xxt}+2\kappa u_x+3uu_x-2u_xu_{xx}-uu_{xxx}=0, \quad \kappa\geq 0.
\end{aligned}
	\end{equation}
	
	The Camassa-Holm equation can also be written in the system form:	
	\begin{equation}\label{CHsys}
		\bigg\{\begin{aligned}
			&u_t + u u_x-\phi_x = 0,\\
			&\phi_{xx}=\phi+2\kappa u +u^2+\frac{1}{2}u_x^2,
		\end{aligned}\bigg.
	\end{equation}
and the Burgers-Poisson equation \eqref{BPsys} is recovered when $\kappa=\frac{1}{2}$ and neglecting the two quadratic terms in the second equation.
	
	While the Camassa-Holm equation is formally integrable (see \cite{MR636470}), this does not seem to be the case of the Fornberg-Whitham equation, see \cite{MR2171998}.
	
	\end{remark}

	Since \eqref{eq:1.1}  is a skew-adjoint perturbation of the Burgers equation,   one easily checks by standard energy methods  that the associated Cauchy problem  is locally well-posed in $H^s(\R), s>3/2,$ so that the nonlocal dispersive term does not allow to enlarge the space of resolution for the Cauchy problem of  the Burgers equation. \footnote{Note that ill-posedness of \eqref{eq:1.1} in $H^{3/2}(\R)$ seems to be an open question. See \cite{MR3188389} for a proof of this result for the Burgers equation.}We will show that it does not prevent the wave breaking phenomena (shock formation).
	
	We say that the solution of \eqref{eq:1.1} exhibits wave breaking (shock formation)
if there exists some $T > 0$ such that
\begin{equation*}
\begin{aligned}
|u(x, t)|<\infty, \quad \text{for}\ x \in \mathbb{R}\ \text{and}\ t \in[0, T),
\end{aligned}
\end{equation*}
but
\begin{equation*}
\begin{aligned}
\sup _{x \in \mathbb{R}}|\partial_{x} u(x, t)| \longrightarrow+\infty,\quad  \text{as} \ t \rightarrow T^-.
\end{aligned}
\end{equation*}	
	
		Results concerning the wave breaking of solutions to the Fornberg-Whitham equation \eqref{FW} were obtained in \cite{Deng, MR2617153, MR3710682, MR4270781, MR4333381}. We refer to \cite{MR4270781} for a review of various issues concerning the Cauchy problem for the Fornberg-Whitham equation.  
		\begin{remark}
		A wave breaking for some solutions of a Fornberg-Whitham equation perturbed by a nonlocal commutator type term is proven in \cite{ADG}.
		\end{remark}
		Our aim in the present  paper is to show that the solutions to the generalized Fornberg-Whitham equation \eqref{eq:1.1} can exhibit wave breaking for initial data with large slope, thus extending the similar known results for the Fornberg-Whitham equation. We will also comment on the dispersive behavior of \eqref{eq:1.1}, in particular on the existence of solitary waves using the fact that it reduces to the KdV equation in the long wave limit, and on linear dispersive estimates.

{\bf{Notations.}}
Let \(\mathcal{F}(g)\) or \(\widehat{g}\) be the Fourier transform of a Schwartz function \(g\) whose formula is given by 
\begin{equation*}
\begin{aligned}
\mathcal{F}(g)(\xi)=\widehat{g}(\xi):=\frac{1}{\sqrt{2\pi}}\int_{\mathbb{R}}g(x)e^{-\mathrm{i}x\xi}\,d x
\end{aligned}
\end{equation*}
with inverse
\begin{equation*}
\begin{aligned}
\mathcal{F}^{-1}(g)(x)=\frac{1}{\sqrt{2\pi}}\int_{\mathbb{R}}g(\xi)e^{\mathrm{i}x\xi}\, d \xi,
\end{aligned}
\end{equation*}
and by \(m(\partial_x)\) the Fourier multiplier with symbol \(m\) via the relation 
\begin{equation*}
\begin{aligned}
\mathcal{F}\big(m(\partial_x)g\big)(\xi)=m(\mathrm{i}\xi)\widehat{g}(\xi).
\end{aligned}
\end{equation*}

Take \(\varphi\in C_0^\infty(\R)\) satisfying \(\varphi(\xi)=1\) for \(|\xi|\leq 1\) and \(\varphi(\xi)=0\) when \(|\xi|>2\), and
let 
\begin{equation*}
\begin{aligned}
\psi(\xi)=\varphi(\xi)-\varphi(2\xi),\quad \psi_j(\xi)=\psi(2^{-j}\xi),\quad \varphi_j(\xi)=\varphi(2^{-j}\xi),
\end{aligned}
\end{equation*}
we then may define the  
Littlewood-Paley projections \(P_j,P_{\leq j},P_{> j}\)  via 
\begin{equation*}
\begin{aligned}
\widehat{P_jg}(\xi)=\psi_j(\xi)\widehat{g}(\xi),\quad \widehat{P_{\leq j}g}(\xi)=\varphi_j(\xi)\widehat{g}(\xi),\quad P_{> j}=1-P_{\leq j},
\end{aligned}
\end{equation*}
and also \(P_{\sim},P_{\lesssim j},P_{\ll j}\) by 
\begin{equation*}
\begin{aligned}
P_{\sim j}=\sum_{2^k\sim 2^j}P_k, \quad P_{\lesssim j}=\sum_{2^k\leq 2^{j+C}}P_k,\quad P_{\ll j}=\sum_{2^k\ll 2^j}P_k,
\end{aligned}
\end{equation*}
and the obvious notation for \(P_{[a,b]}\).
We will also denote \(g_j=P_jg, g_{\lesssim j}=P_{\lesssim j}g\), and so on, for convenience.

The notation \(C\)  always denotes a nonnegative universal constant which may be different from line to line but is
independent of the parameters involved. Otherwise, we will specify it by  the notation \(C(a,b,\dots)\).
We write \(g\lesssim h\) (\(g\gtrsim h\)) when \(g\leq  Ch\) (\(g\geq  Ch\)), and \(g \sim h\) when \(g \lesssim h \lesssim g\).
We also write \(\sqrt{1+x^2}=\langle x\rangle\) and \(\|g\|_{H^{1,1}}=\|\langle x\rangle g\|_{H^1}\) for simplicity.

	\section{Main results}	
	
	\subsection{The case \(p=1\)}
In this case, we show that the solution to \eqref{eq:1.1} can exhibit wave breaking for 
$s \in(2/5, \infty)$.
	\begin{theorem}\label{th:2.1}
		Let $s \in(2/5, 1)$. If $u_{0} \in H^{3}(\mathbb{R})$ satisfies	the following slope conditions:
		\begin{equation}\label{eq:2.1}
			\begin{aligned}
				&\delta^{2} \big[\inf_{x \in \mathbb{R}} u_{0} ^{\prime}(x)\big]^{2} > C\big(\|u_{0}\|_{H^{3}} +  C_{1} + C_{1}^{\frac{1}{3}}\|u_{0}^{\prime \prime \prime}\|_{L^{2}}^{\frac{2}{3}}\big),\\
				&(1-\delta)^{2}\big[-\inf_{x \in \mathbb{R}} u_{0} ^{\prime}(x)\big] > C (1+C_0^{-1}C_{1}),\\
				&(1-\delta)^{3} \big[-\inf_{x \in \mathbb{R}} u_{0} ^{\prime}(x)\big] > C \big(1+C_{1}^{-\frac{2}{3}}\|u_{0}^{\prime \prime \prime}\|_{L^{2}}^{\frac{2}{3}}\big),
			\end{aligned}
		\end{equation}
		where \(\delta\in (0,1)\) is a small number, and $C_{0},C_{1}>0$ satisfy
		\begin{equation}\label{eq:2.2}
			\|u_{0}\|_{L^{\infty}} \leq C_{0}/2, \quad \|u_{0}^{\prime}\|_{L^{\infty}} \leq C_{1}/2.
		\end{equation}
		Then the solution $u(t, x)$ to \eqref{eq:1.1} exhibits wave breaking at  $T > 0$ with
		\begin{equation}\label{eq:2.3}
			(1+\delta)^{-1} \big[-\inf_{x \in \mathbb{R}} u_{0} ^{\prime}(x)\big]^{-1}<T<(1-\delta)^{-2} \big[-\inf_{x \in \mathbb{R}} u_{0} ^{\prime}(x)\big]^{-1}.
		\end{equation}
	\end{theorem}

	\begin{theorem}\label{th:2.2}
		Let $s \in [1,\infty)$.
		If  $u_{0} \in H^{2}(\mathbb{R})$ satisfies the following slope conditions:
		\begin{equation}\label{eq:2.4}
			\begin{aligned}
				&\delta^{2} \big[-\inf_{x \in \mathbb{R}} u_{0} ^{\prime}(x)\big]^{2} >C \big( \|u_{0}\|_{H^{2}} + C_{1} +\|u_{0}^{\prime \prime}\|_{L^{2}}\big),\\
				&(1-\delta)^{2}\big[-\inf_{x \in \mathbb{R}} u_{0} ^{\prime}(x)\big] > C C_{0}^{-1}\|u_{0}^{\prime}\|_{L^{2}},\\
				&(1-\delta)^{3} \big[-\inf_{x \in \mathbb{R}} u_{0} ^{\prime}(x)\big] >C  \big(1+C_{1}^{-1}\|u_{0}^{\prime  \prime}\|_{L^{2}}\big),
			\end{aligned}
		\end{equation}
		where \(\delta\in (0,1)\) is a small number, and  $C_{0},C_{1}>0$ satisfy \eqref{eq:2.2}. Then the solution $u(t,x)$ to \eqref{eq:1.1} exhibits wave breaking at $T > 0$ with \eqref{eq:2.3}.
	\end{theorem}
	
	\begin{remark}
		There exists a class of initial data $u_{0}$ satisfying the conditions $\eqref{eq:2.1}_1$-$\eqref{eq:2.1}_3$ in Theorem \ref{th:2.1}. Indeed, for  any given $\phi \in H^3(\mathbb{R})$ with $\inf_{x \in \mathbb{R}} \phi ^{\prime}(x)<0$, set 
 \begin{equation*}
\begin{aligned}
&u_{0}=\lambda \phi,\\
&C_{0}=2 \lambda\|\phi\|_{L^{\infty}}, \\
&C_{1}=2 \lambda\|\phi^{\prime}\|_{L^{\infty}}.
\end{aligned}
\end{equation*}
Choosing  $\lambda>0$ sufficiently large, one can easily check that $u_{0}$  satisfies $\eqref{eq:2.1}_1$-$\eqref{eq:2.1}_3$ by comparing the powers of  $\lambda$  on both sides of each inequality. For example, $u_{0}(x)=\lambda e^{-x^2}$ with $\lambda>0$ sufficiently large satisfies $\eqref{eq:2.1}_1$-$\eqref{eq:2.1}_3$. 
	\end{remark}

	\begin{remark}
		The conditions $\eqref{eq:2.4}_1$-$\eqref{eq:2.4}_3$ lower the requirement on the regularity of the initial data \(u_0\) of $\eqref{eq:2.1}_1$-$\eqref{eq:2.1}_3$ from \(H^3\) to \(H^2\) since the dispersion effect of the Bessel potential is much weaker when \(s\) is larger. 	
	\end{remark}

	\begin{remark}
	Oh and Pasqualotto  \cite{2021arXiv210707172O} obtained precise wave breaking information for the solution to \eqref{eq:1.1} when $p=1$ and $s \in (0, 1]$ with more delicated assumptions on the initial data by the modulation theory (see also Yang's work \cite{MR4321245} on the Burgers-Hilbert equation).  
	\end{remark}

	\subsection{The case \(p>1\)}
	In this case, we show that the solution to \eqref{eq:1.1} can exhibit wave breaking for all
$s \in(0, \infty)$.
	\begin{theorem}\label{th:2.3}
		Let $s \in(0, 1)$. Suppose  $\bar{x}_{1}$ and  $\bar{x}_{2}$ are the largest and smallest numbers such that  $\{x: u_{0}^{\prime}(x)<0\} \subset[\bar{x}_{1},  \bar{x}_{2}]$.  If  $u_{0} \in H^{3}(\mathbb{R})$  satisfies the following slope conditions:
		{\small \begin{equation}\label{eq:2.5}
				\begin{aligned}
					&\delta^{2} \big[-\inf_{x \in \mathbb{R}} u_{0} ^{\prime}(x)\big]^{2} >C\bigg[\|u_{0}\|_{H^{3}} +  C_{1}
					+\bigg(\frac{A^{p-1}}{2B^{p-1}}\bigg)^{-\frac{7B^{p-1}}{2pA^{2p-2}}}\|u_{0}^{\prime \prime \prime}\|_{L^{2}}\bigg],	\\
					&(1-\delta)^{2}\big[-\inf_{x \in \mathbb{R}} u_{0} ^{\prime}(x)\big] > C(1 + C_{1}C_{0}^{-1}),\\
					&(1-\delta)^{3} \big[-\inf_{x \in \mathbb{R}} u_{0} ^{\prime}(x)\big] >C\bigg[ 1 +\bigg(\frac{A^{p-1}}{2B^{p-1}}\bigg)^{-\frac{7B^{p-1}}{pA^{2p-2}}}C_{1}^{-1}\|u_{0}^{\prime \prime \prime}\|_{L^{2}} \bigg],
				\end{aligned}
		\end{equation}}
		and local amplitude conditions:
		\begin{equation}\label{eq:2.6}
			\begin{aligned}
				&u_{0}(x)<B-C(1-\delta)^{-2} \big[-\inf_{x \in \mathbb{R}} u_{0}^{\prime}(x)\big]^{-1}(C_{0}+ C_{1}),\\
				&u_{0}(x)>A+C(1-\delta)^{-2} \big[-\inf_{x \in \mathbb{R}} u_{0}^{\prime}(x)\big]^{-1}(C_{0}+ C_{1})
			\end{aligned}
		\end{equation}
		for all $x\in[\bar{x}_{1}, \bar{x}_{2}]$,
		where \(\delta\in (0,1)\) is a small number, and $C_{0}, C_{1}>0$ satisfy \eqref{eq:2.2}, and $A, B > 0$ satisfy
		\begin{equation}\label{eq:2.7}
			\begin{array}{l}
				A^{2p-2}>8\delta B^{2p-2}, \quad 4pA^{2p-2}>7B^{p-1},   \\
				B>A+C(1-\delta)^{-2} \big[-\inf_{x \in \mathbb{R}} u_{0}^{\prime}(x)\big]^{-1}(C_{0}+ C_{1}).
			\end{array}
		\end{equation}
		Then the solution $u(t,  x)$ to \eqref{eq:1.1} exhibits wave breaking at  $T > 0$ with
		{\small \begin{equation}\label{eq:2.8}
\begin{aligned}
				\frac{1}{pB^{p-1}+\delta} \frac{1}{[-\inf_{x \in \mathbb{R}} u_{0}^{\prime}(x)]}<T<\frac{1}{(A^{p-1}B^{1-p}-\delta)(pA^{p-1}-\delta)} \frac{1}{[-\inf_{x \in \mathbb{R}} u_{0} ^{\prime}(x)]}.
\end{aligned}
		\end{equation}}
	\end{theorem}
	
	\begin{theorem}\label{th:2.4}
		Let $s \in [1, \infty)$. Suppose  $\bar{x}_{2}$  are the largest and smallest numbers such that  $\{x: u_{0}^{\prime}(x)<0\} \subset[\bar{x}_{1}, \bar{x}_{2}]$. If  $u_{0}$  satisfies the slope conditions:
		{\small \begin{equation*}
				\begin{aligned}
					&\delta^{2} \big[-\inf_{x \in \mathbb{R}} u_{0} ^{\prime}(x)\big]^{2} > C \bigg[ \|u_{0}\|_{H^{2}} +  \bigg(C_{1} +\bigg(\frac{A^{p-1}}{2B^{p-1}}\bigg)^{-\frac{5B^{p-1}}{2pA^{2p-2}}}\|u_{0}^{\prime \prime }\|_{L^{2}} \bigg)\bigg],\\
					&(1-\delta)^{2}\big[-\inf_{x \in \mathbb{R}} u_{0} ^{\prime}(x)\big] >  C\bigg(\frac{A^{p-1}}{2B^{p-1}}\bigg)^{-\frac{B^{p-1}}{2pA^{2p-2}}} C_{0}^{-1}\|u_{0}^{\prime}\|_{L^{2}},\\
					&(1-\delta)^{3} \big[-\inf_{x \in \mathbb{R}} u_{0} ^{\prime}(x)\big] >C\bigg(\frac{A^{p-1}}{2B^{p-1}}\bigg)^{-\frac{5B^{p-1}}{2pA^{2p-2}}}C_{1}^{-1}\|u_{0}^{\prime \prime }\|_{L^{2}},
				\end{aligned}
		\end{equation*}}
		and local amplitude conditions:
		\small{\begin{equation*}
				\begin{aligned}
					&u_{0}(x)<B-C (1-\delta)^{-2} \big[-\inf_{x \in \mathbb{R}} u_{0} ^{\prime}(x)\big]^{-1}\bigg(\frac{A^{p-1}}{2B^{p-1}}\bigg)^{-\frac{B^{p-1}}{2pA^{2p-2}}}\|u_{0}^{\prime}\|_{L^{2}},\\
					&u_{0}(x)>A+C (1-\delta)^{-2} \big[-\inf_{x \in \mathbb{R}} u_{0} ^{\prime}(x)\big]^{-1}\bigg(\frac{A^{p-1}}{2B^{p-1}}\bigg)^{-\frac{B^{p-1}}{2pA^{2p-2}}}\|u_{0}^{\prime}\|_{L^{2}}
				\end{aligned}
		\end{equation*}}
		for all $x\in[\bar{x}_{1}, \bar{x}_{2}]$,
		where \(\delta\in (0,1)\) is a small number, and $C_{0}, C_{1}>0$ satisfy \eqref{eq:2.2}, and $A, B > 0$ satisfy
		\begin{equation*}
			\begin{array}{l}
				A^{2p-2}>8\delta B^{2p-2}, \quad 4pA^{2p-2}>7B^{p-1},   \\
				B>A+C(1-\delta)^{-2} \big[-\inf _{x \in \mathbb{R}} u_{0}^{\prime}(x)\big]^{-1}\bigg(\frac{A^{p-1}}{2B^{p-1}}\bigg)^{-\frac{B^{p-1}}{2pA^{2p-2}}}\|u_{0}^{\prime}\|_{L^{2}}.
			\end{array}
		\end{equation*}
		Then the solution $u(t, x)$ to \eqref{eq:1.1}  exhibits wave breaking at $T > 0$ with \eqref{eq:2.8}.
	\end{theorem}
	\begin{remark}
		It should be pointed out that the interval $[\bar{x}_{1}, \bar{x}_{2}]$ can be replaced by a larger but finite interval in the local amplitude conditions \eqref{eq:2.6}. Otherwise, the local amplitude condition will become a global one, which means $u_{0}\notin L^{2}(\mathbb{R})$ and thus contradicts $u_{0}\in H^{3}(\mathbb{R})$. More importantly, if the latter case happened, then $u_{0}$ is bounded below by a positive constant on the entire line which is physically strange since $u_{0}$ is the initial elevation. In the classical water waves models (such as the KdV equation), the solution is assumed to tend to zero at infinity.
	\end{remark}

	\section{Preliminaries}
	
	First, we list some basic properties of the Bessel potential.
	\begin{lemma}\label{lem:1}
		There exists some constant $C$ only depending on $s$ such that
		\begin{equation}\label{eq:3.1}
			\begin{aligned}
				&G_{s}(x) \leq C\left\{\begin{array}{ll}
					\frac{1}{|x|^{1-s}} &\quad \mathrm{for}\  |x|\leq 1\ \mathrm{and}\ \ 0<s<1, \\
					\log \frac{1}{|x|}+1 &\quad \mathrm{for}\  |x|\leq 1\ \mathrm{and}\ \ s=1,\\
					1 &  \quad \mathrm{for}\  |x|\leq 1\ \mathrm{and}\ \ s>1,\\
					|x|^{\frac{s-2}{2}} e^{-|x|}&  \quad \mathrm{for}\  |x|> 1\ \mathrm{and}\ \ s>0
				\end{array}\right.
			\end{aligned}
		\end{equation}
		and
		\begin{equation}\label{eq:3.2}
			\begin{aligned}
				&|G_{s}^{\prime }(|x|)|
				\leq C\left\{\begin{array}{ll}
					\frac{1}{|x|^{2-s}} &\quad \mathrm{for}\  |x|\leq 1\ \mathrm{and}\ \ 0<s<2,\\
					1 & \quad \mathrm{for}\  |x|\leq 1\ \mathrm{and}\ \ s\geq 2,\\
					|x|^{\frac{s-2}{2}} e^{-|x|}&  \quad \mathrm{for}\  |x|> 1\ \mathrm{and}\ \ s>0.
				\end{array}\right.
			\end{aligned}
		\end{equation}
		In particular, one has
		\begin{equation}\label{eq:3.3}
			\int_{1}^{\infty} |G_{s}^{\prime}(|x|)| \mathrm{~d}x \leq C\quad \mathrm{for}\ s>0
		\end{equation}
		and	
		\begin{equation}\label{eq:3.4}
			\int_{\eta}^{\infty} G_{s}^{2}(x) \mathrm{~d}x \leq C\quad \mathrm{for}\ s\geq 1 \ \mathrm{and}\ \eta \in (0,1].
		\end{equation}
		Here \(C\) in \eqref{eq:3.4} does not depend on \(\eta\). 		
	\end{lemma}
	\begin{proof}
		Clearly one can write
		{\small \begin{equation}\label{eq:3.5}
				G_{s}(x)=c_1(s)|x|^{\frac{s-1}{2}}K_{\frac{1-s}{2}}(|x|)
		\end{equation}}
		and
		{\small \begin{equation}\label{eq:3.6}
				G_{s}^{\prime}(|x|)=-c_2(s)|x|^{\frac{s-1}{2}}K_{\frac{3-s}{2}}(|x|),
		\end{equation}}
		where $K_\nu$ is known as a modified Bessel function of the third kind given by
		\begin{equation*}
			\begin{aligned}
				K_\nu(r)=\dfrac{1}{2}\Big(\frac{r}{2}\Big)^\nu\int_0^\infty\dfrac{e^{-t-\frac{r^2}{4t}}}{t^{\nu+1}}\, \mathrm{~d}t\quad\text{for } r>0.
			\end{aligned}
		\end{equation*}	
		$K_\nu(r)$ is analytic on
		$r$ except at $r = 0$ and  even on  $\nu$, and has the following asymptotic formulae (see \cite{MR1411441, MR0374877}):
		\begin{equation}\label{eq:3.7}
			\begin{aligned}
				&K_{\nu}(r) \leq C(\nu)\left\{\begin{array}{ll}
					 r^{-\nu}, &  0<r\leq 1,\ \nu>0 \\
					\log (1 / r)+1, &  0<r\leq 1,\ \nu=0 \\
					r^{-1/2} e^{-r}, &  r>1,\ \nu\geq 0.
				\end{array}\right.
			\end{aligned}
		\end{equation}
		
		The estimates \eqref{eq:3.1}-\eqref{eq:3.2} follow from \eqref{eq:3.5}-\eqref{eq:3.7} directly. The estimate \eqref{eq:3.3} can be easily verified by \eqref{eq:3.2}.
		In order to estimate \eqref{eq:3.4}, we spilt the integral as follows:
		\begin{equation}\label{eq:3.8}
			\int_{\eta}^{\infty} G_{s}^{2}(x) \mathrm{~d}x =\int_{\eta}^{1} G_{s}^{2}(x) \mathrm{~d}x+\int_{1}^{\infty} G_{s}^{2}(x) \mathrm{~d}x.
		\end{equation}
		It  needs only to take care of the first integral on the RHS of \eqref{eq:3.8} when \(s=1\). Indeed, noticing that
		\begin{equation}\label{eq:3.9}
			\begin{aligned}
				\log\bigg(\frac{1}{|y|}\bigg)\leq \frac{C}{|y|^{\frac{2}{5}}}\quad \text{for} \ 0<|y|<1,
			\end{aligned}
		\end{equation}
		one may estimate
		\begin{equation*}
			\int_{\eta}^{1} G_{1}^{2}(x) \mathrm{~d}x\leq C \int_{|y|<1}\bigg(\frac{1}{|y|^{\frac{2}{5}}} +1\bigg)^{2}\mathrm{~d}y\leq C.
		\end{equation*}
		
	\end{proof}	
	
	Next, we recall the standard Gagliardo-Nirenberg interpolation inequality.
	\begin{lemma}\label{lem:2}
		Let $1 \leq q,~ r \leq \infty,~ j,~ m \in N$  with  $j / m \leq \theta \leq 1$. If
		$$\frac{1}{p}=j+\theta\bigg(\frac{1}{r}-m\bigg)+\frac{1-\theta}{q},$$
		then
		\begin{equation*}
			\|\partial_{x}^{j} u\|_{L^{p}} \leq C\|\partial_{x}^{m} u\|_{L^{r}}^{\theta}\|u\|_{L^{q}}^{1-\theta},
		\end{equation*}
		where the constant \(C\) depends only on  $j, m, r, p, q, \theta $.
	\end{lemma}

	\section{Reformulation in Lagrangian coordinates}

	It is standard to show that there exists some positive \(T\) such that \eqref{eq:1.1} admits a solution $u\in C([0,T); H^3(\mathbb{R}))$, see for example \cite{MR533234}. In what follows, we will assume $T$ is the maximal
	existence time of the solution \(u\).
	
	Denote by $X(t , x)$ the position of the particle \(x\) at time \(t\)
	\begin{equation*}
		\begin{aligned}
			\frac{\mathrm{d} X}{\mathrm{d} t}(t , x)=u^p(t, X(t , x)) \quad \text { and } \quad X(0 , x)=x.
		\end{aligned}
	\end{equation*}
	Let
	\begin{equation*}
		v_{n}(t , x)=(\partial_{x}^{n-1} u)(t, X(t , x)) \quad  \text{for}\  n=1,2.
	\end{equation*}
	Then it follows from \eqref{eq:1.1} that
	\begin{equation}\label{eq:4.1}
		\frac{\mathrm{d} v_{1}}{\mathrm{d} t} = -K_{1}^{s}(t, x),
	\end{equation}
	where
	\begin{equation}\label{eq:4.2}
		K_{1}^{s}(t,x)  =\int_{\mathbb{R}} G_{s}(y) \partial_{x} u(t, X(t, x)-y) \mathrm{~d} y
	\end{equation}
	and
	\begin{equation}\label{eq:4.3}
		\frac{\mathrm{d} v_{2}}{\mathrm{d} t} = -pv_{1}^{p-1}v_{2}^2 - K_{2}^{s}(t, x),
	\end{equation}
	in which
	\begin{equation}\label{eq:4.4}
		K_{2}^{s}(t, x) =\int_{\mathbb{R}} G_{s}(y) \partial_{x}^{2} u(t, X(t, x)-y) \mathrm{~d} y.
	\end{equation}
	
	Set
	\begin{equation}\label{eq:4.5}
		m(t)=\inf _{x \in \mathbb{R}} v_{2}(t , x)=\inf _{x \in \mathbb{R}}(\partial_{x} u)(x, t)=\colon m(0) q(t)^{-1}.
	\end{equation}
	To prove Theorem \ref{th:2.1} -\ref{th:2.4}, it suffices to show \(q(t)\rightarrow 0\) as \(t\rightarrow T^-\), whose key ingredient is to prove the following:
	\begin{equation}\label{eq:4.6}
		|K_{2}^{s}(t , x)|<\delta^{2} m^{2}(t) \quad \text{for}\ (t,x) \in [0,T)\times \mathbb{R}.
	\end{equation}

	\section{Proof of Theorem \ref{th:2.1}}

	Note that \(p=1\). In what follows, we always assume $\delta\in (0,1)$ is a sufficiently small number.

	We first check that \eqref{eq:4.6} holds at $t=0$ by the assumption $\eqref{eq:2.1}_1$. To this end, one writes $K_{2}^{s}(0, x)$ as follows:
	\begin{equation}\label{eq:5.1}
		K_{2}^{s}(0,x)
		=\underbrace{\int_{|y|<1} G_{s}(y) u_{0}^{\prime \prime}(x-y) \mathrm{~d} y}_{I_1}+\underbrace{\int_{|y| \geq 1} G_{s}(y) u_{0}^{\prime \prime}(x-y) \mathrm{~d} y}_{I_2}.
	\end{equation}
	Applying  $\eqref{eq:3.1}_1$ to \(I_1\) yields
	\begin{equation}\label{eq:5.2}
		|I_1| \leq C \|u_{0}^{\prime \prime}\|_{L^{\infty}}\bigg|\int_{|y|<1}\frac{1}{|y|^{1-s}} \mathrm{~d} y\bigg| \leq C\|u_{0}\|_{H^{3}},
	\end{equation}
	where the Sobolev embedding \(H^1(\mathbb{R})\hookrightarrow L^\infty(\mathbb{R})\) has been used.
	To handle \(I_2\), one may integrate by parts to deduce
	\begin{equation}\label{eq:5.3}
		\begin{aligned}
			|I_2|
			&\leq|G_{s}(1)[u_{0}^{\prime}(-1-y)-u_{0}^{\prime}(1-y)]|
			+\bigg|\int_{|y| \geq 1} G_{s}^{\prime}(y) u_{0}^{\prime}(x-y) \mathrm{~d} y\bigg| \\
			&\leq C\|u_{0}^{\prime}\|_{L^{\infty}}
			\leq C C_{1},
		\end{aligned}
	\end{equation}
	where one has used \eqref{eq:3.3}.
	
	It follows from $\eqref{eq:2.1}_1$ and \eqref{eq:5.1}-\eqref{eq:5.3} that
	\begin{equation}\label{eq:5.4}
		|K_{2}^{s}(0,x) |\leq C(\|u_{0}\|_{H^{3}} + C_{1}) < \delta^{2} m^{2}(0)\quad \text{for}\ x \in \mathbb{R}.
	\end{equation}

	Next, we will show \eqref{eq:4.6} for $t\neq 0$ by the assumptions $\eqref{eq:2.1}_1$-$\eqref{eq:2.1}_3$ together with  an argument of contradiction.
	Suppose that \eqref{eq:4.6} is not true, then  there exist some \(T_{1} \in(0, T)\) and \( x_{0} \in \mathbb{R}\) such that
	\begin{equation}\label{eq:5.5}
		\boxed{|K_{2}^{s}(T_{1}, x_{0})|=\delta^{2} m^{2}(T_{1})}.
	\end{equation}
	Thus, without loss of generality, we may assume by continuity that
	\begin{equation*}
		|K_{2}^{s}(t,x)| \leq \delta^{2} m^{2}(t) \quad \text{for}\ (t,x) \in (0,T_{1}]\times \mathbb{R}.
	\end{equation*}
	
	For \((t,x) \in (0,T_{1}]\times \mathbb{R}\), set
	\begin{equation}\label{eq:5.6}
		\Sigma_{\delta}(t)=\{x \in \mathbb{R}: v_{2}(t , x) \leq(1-\delta) m(t)\}
	\end{equation}
	and
	\begin{equation}\label{eq:5.7}
		v_{2}(t, x)\bigg(=\frac{v_{1}(0)}{1+v_{1}(0) \int_{0}^{t}[1+v_{1}^{-2} K_{2}^{s}(\tau, x)] \mathrm{~d} \tau}\bigg)=\colon  m(0) r^{-1}(t, x),
	\end{equation}
	where the first equality in \eqref{eq:5.7} comes from solving \eqref{eq:4.3}.

	Then, one has the following lemmas, whose proofs are quite similar with that in \cite{MR1261868, MR4409228}, so we omit it.
	
	\begin{lemma}~\label{lem:3}
		For fixed \(\delta\), the set $\Sigma_{\delta}(t)$ is decreasing in \(t\), namely
		$\Sigma_{\delta}(t_{2}) \subset \Sigma_{\delta}(t_{1})$  whenever  $0 \leq t_{1} \leq t_{2} \leq T_{1} $.
	\end{lemma}
	
	\begin{lemma}\label{lem:4}
		We have
		\begin{equation}\label{eq:5.8}
			(1+\delta) m(0) \leq\frac{\mathrm{~d} }{\mathrm{~d}t}r(t, x) \leq(1-\delta) m(0) \quad \mathrm{for}\ x \in\Sigma_{\delta}(T_1),
		\end{equation}
		\begin{equation}\label{eq:5.9}
			q(t) \leq r(t, x) \leq\frac{1}{1-\delta} q(t) \quad \mathrm{for}\ x \in\Sigma_{\delta}(T_1)
		\end{equation}
		and
		\begin{equation}\label{eq:5.10}
			0<q(t) \leq 1.
		\end{equation}
	\end{lemma}

	\begin{lemma}\label{lem:5} We have 		
		\begin{equation}\label{eq:5.11}
			\int_{0}^{t} q(\tau)^{-\gamma}\mathrm{~d} \tau \leq -(1- \delta)^{-(\gamma +1)}(1-\gamma )^{-1} m^{-1}(0)\big[(1-\delta)^{\gamma -1}-q(t)^{1-\gamma}\big],
		\end{equation}		
		where  $\gamma\in (0,1)\cup (1,\infty)$, and
		
		\begin{equation}\label{eq:5.12}
			\int_{0}^{t} q(\tau)^{-1}\mathrm{~d} \tau \leq-(1- \delta)^{-2} m^{-1}(0)[-\log (1- \delta)-\log q(t)].
		\end{equation}
		
	\end{lemma}

	We first claim that
	\begin{equation}\label{eq:5.13}
		\|v_{1}(t)\|_{L^{\infty}}=\|u(t)\|_{L^{\infty}}<C_{0} \quad  \text{for all} \ t \in[0, T_{1}]
	\end{equation}
	and
	\begin{equation}\label{eq:5.14}
		\|v_{2}(t)\|_{L^{\infty}}=\|\partial_{x} u(t)\|_{L^{\infty}}<C_{1} q(t)^{-1} \quad  \text{for all} \ t \in[0, T_{1}],
	\end{equation}
	where  $C_{0}\ \text{and}\ C_{1}$  satisfy \eqref{eq:2.2}.
	
	First, when \(t=0\), observe that
	\begin{equation*}
		\|v_{1}(0)\|_{L^{\infty}}=\|u_{0}\|_{L^{\infty}} < C_{0}
	\end{equation*}
	and
	\begin{equation*}
		\|v_{2}(0)\|_{L^{\infty}}=\|u_{0}^{\prime}\|_{L^{\infty}} < C_{1} q(t)^{-1}.
	\end{equation*}
	Then, it remains to show \eqref{eq:5.13} and \eqref{eq:5.14} when \(t\neq 0\), which will be achieved by the contradiction argument again.  	
	Assume that there exists  $T_{2}\in (0,T_{1}]$ such that \eqref{eq:5.13} and \eqref{eq:5.14} hold for all $t \in[0, T_{2})$, but either of them fails at  $t=T_{2}$, that is
	\begin{equation}\label{eq:5.15}
		\boxed{\text{either}\ \|u(T_{2})\|_{L^{\infty}}=C_{0}\   \text{or}\  \|\partial_{x} u(T_{2})\|_{L^{\infty}}=C_{1} q^{-1}(t)}.
	\end{equation}
	Hence, by continuity, one has
	\begin{equation}\label{eq:5.16}
		\|v_{1}(t)\|_{L^{\infty}}=\|u(t)\|_{L^{\infty}} \leq C_{0} \quad  \text{for} \  t \in[0, T_{2}]
	\end{equation}
	and
	\begin{equation}\label{eq:5.17}
		\|v_{2}(t)\|_{L^{\infty}}=\|\partial_{x} u(t)\|_{L^{\infty}} \leq C_{1} q(t)^{-1} \quad  \text{for} \ t \in[0, T_{2}].
	\end{equation}
	
	\underline{\bf{Estimates on \(K_{1}^{s}(t, x)\)}.}  We split the integral \eqref{eq:4.2} into two parts as follows:
	{\small \begin{equation}\label{eq:5.18}
			K_{1}^{s}(t,x)=\underbrace{\int_{|y| \leq \eta} G_{s}(y) \partial_{x} u(t, X(t, x)-y) \mathrm{~d} y}_{I_{3}}+\underbrace{\int_{|y|>\eta} G_{s}(y) \partial_{x} u(t, X(t, x)-y) \mathrm{~d} y}_{I_{4}},
	\end{equation}}
	
	\noindent	where $\eta=\eta(t) \in(0,1]$ will be determined later. By $\eqref{eq:3.1}_1$ and \eqref{eq:5.17}, one has
	\begin{equation}\label{eq:5.19}
		|I_{3}| \leq C \|v_{2}\|_{L^{\infty}} \int_{|y| \leq \eta} \frac{1}{|y|^{1-s}} \mathrm{~d} y \leq C\eta^{s}\|v_{2}\|_{L^{\infty}}\leq C C_{1} \eta^{s} q(t)^{-1}.
	\end{equation}
	By $\eqref{eq:3.1}_1$, $\eqref{eq:3.2}_1$,  \eqref{eq:3.3} and \eqref{eq:5.16}, one integrates by parts to find that
	\begin{equation}\label{eq:5.20}
		\begin{aligned}
			|I_{4}| &\leq  |G_{s}(\eta)[u(t, X(t, x)-\eta)-u(t, X(t, x)+\eta)]| \\
			&\quad +|\int_{\eta<|y| \leq 1} G_{s}^{\prime}(y) u(t, X(t, x)-y) \mathrm{~d} y| \\
			&\quad +|\int_{|y|>1} G_{s}^{\prime}(y) u(t, X(t, x)-y) \mathrm{~d} y| \\
			&\leq  C [\eta^{s-1}\|v_{1}\|_{L^{\infty}}+(\eta^{s-1}-1)\|v_{1}\|_{L^{\infty}}+  \|v_{1}\|_{L^{\infty}}] \\
			&\leq  C  \eta^{s-1}\|v_{1}\|_{L^{\infty}} \leq C C_{0}\eta^{s-1}.
		\end{aligned}
	\end{equation}
	By choosing  $\eta=q(t)$,  it follows from \eqref{eq:5.19} and \eqref{eq:5.20} that
	\begin{equation}\label{eq:5.21}
		|K_{1}^{s}(t,x)| \leq C (C_{0}+C_{1}) q(t)^{s-1}\quad \text{for}\ (t,x) \in (0,T_{2}]\times \mathbb{R}.
	\end{equation}

	\underline{\bf{Estimates on \(v_{1}(t,x)\)}.}
	By \eqref{eq:5.16} and \eqref{eq:5.21}, one uses \eqref{eq:4.1} and \eqref{eq:5.11} to deduce that
	\begin{equation}\label{eq:5.22}
		\begin{aligned}
			|v_{1}(t,x)|
			& \leq  \|u_{0}\|_{L^{\infty}}+C(C_{0}+C_{1}) \int_{0}^{t} q(\tau)^{s-1} \mathrm{~d} \tau \\
			& \leq  \frac{1}{2} C_{0}-C(C_{0}+C_{1})(1-\delta)^{-2} m^{-1}(0) \\
			& <C_{0} \quad \text{for}\ (t,x) \in (0,T_{2}]\times \mathbb{R},
		\end{aligned}
	\end{equation}
	where $\eqref{eq:2.1}_2$ has been used in the last inequality.
	
	\underline{\bf{Estimates on \(K_{2}^{s}(t, x)\)}.} Similar to \eqref{eq:5.18}, we also write the integral \eqref{eq:4.4} as follows:
	{\small \begin{equation}\label{K2}
			K_{2}^{s}(t, x)=\underbrace{\int_{|y| \leq \eta} G_{s}(y) \partial_{x}^{2} u(t, X(t, x)-y) \mathrm{~d} y}_{I_{5}}+\underbrace{\int_{|y|>\eta} G_{s}(y) \partial_{x}^{2} u(t, X(t, x)-y) \mathrm{~d} y}_{I_{6}}.
	\end{equation}}
	
	In a similar fashion to \eqref{eq:5.19} and  \eqref{eq:5.20}, one can estimate
	\begin{equation}\label{eq:5.23}
		|I_{5}| \leq C \|\partial_{x}^{2} u\|_{L^{\infty}}  \int_{|y| \leq \eta} \frac{1}{|y|^{1-s}} \mathrm{~d} y \leq C \eta^{s}\|\partial_{x}^{2} u\|_{L^{\infty}}
	\end{equation}
	and
	\begin{equation}\label{eq:5.24}
		\begin{aligned}
			|I_{6}| &\leq  |G_{s}(\eta)[\partial_{x} u(t, X(t, x)-\eta)-\partial_{x} u(t, X(t, x)+\eta)]| \\
			&\quad +\bigg|\int_{\eta<|y| \leq 1} G_{s}^{\prime}(y) \partial_{x} u(t, X(t, x)-y) \mathrm{~d} y\bigg| \\
			&\quad +\bigg|\int_{|y|>1} G_{s}^{\prime}(y) \partial_{x} u(t, X(t, x)-y) \mathrm{~d} y\bigg| \\
			&\leq C \big[\eta^{s-1}\|v_{2}\|_{L^{\infty}}+\big(\eta^{s-1}-1\big)\|v_{2}\|_{L^{\infty}}+\|v_{2}\|_{L^{\infty}}\big] \\
			&\leq  C C_{1} \eta^{s-1} q(t)^{-1}.
		\end{aligned}
	\end{equation}

	In order to get the blow-up rate of \(I_5\), it suffices to estimate  $\|\partial_{x}^{2} u\|_{L^{\infty}}$. Noticing that $\|\partial_{x}^{2} u\|_{L^{\infty}}$ is not contained in \eqref{eq:5.16} or \eqref{eq:5.17}, our idea is to use $\|\partial_{x}^{3} u\|_{L^{2}}$
	to control $\|\partial_{x}^{2} u\|_{L^{\infty}}$. To this end,
	we turn to estimate $\|\partial_{x}^{3} u\|_{L^{2}}$ by energy estimates. Applying $\partial^{3}_{x}$ to \eqref{eq:1.1}, and multiplying it by $\partial^{3}_{x}u$, one obtains
	\begin{equation}\label{eq:5.25}
		\begin{aligned}
			\frac{1}{2} \frac{\mathrm{d}}{\mathrm{d} t} \int_{\mathbb{R}}(\partial_{x}^{3} u)^{2} \mathrm{d} x& = \int_{\mathbb{R}} \partial_{x}^{3} u \int_{\mathbb{R}} G_{s}(x-y) \partial_{y}^{4} u(y)\, \mathrm{d} y \mathrm{d} x \\
			&\quad -\int_{\mathbb{R}}\big[4 \partial_{x} u(\partial_{x}^{3} u)^{2}+3(\partial_{x}^{2} u)^{2} \partial_{x}^{3} u+u \partial_{x}^{4} u \partial_{x}^{3} u\big] \mathrm{~d} x\\
			&=-\frac{7}{2}\int_{\mathbb{R}} \partial_{x} u(\partial_{x}^{3} u)^{2} \mathrm{~d} x,
		\end{aligned}
	\end{equation}
	where one has used the fact that $G_{s}(\cdot)$ is an even	kernel, and the following equalities:
	\begin{equation*}
		\begin{aligned}
			\int_{\mathbb{R}} u \partial_{x}^{4} u
			\partial_{x}^{3} u \mathrm{~d} x  =-\frac{1}{2} \int_{\mathbb{R}} \partial_{x} u(\partial_{x}^{3} u)^{2} \mathrm{~d} x\quad \text{and} \quad
			\int_{\mathbb{R}}(\partial_{x}^{2} u)^{2} \partial_{x}^{3} u \mathrm{~d} x  =0.
		\end{aligned}
	\end{equation*}
	Consequently, it follows from \eqref{eq:5.25} and \eqref{eq:4.5} that
	\begin{equation}\label{eq:5.26}
		\frac{\mathrm{d}}{\mathrm{d} t} \int_{\mathbb{R}}(\partial_{x}^{3} u)^{2} \mathrm{~d} x
		\leq-7 m(0) q(t)^{-1}\|\partial_{x}^{3} u\|_{L^{2}}^{2},
	\end{equation}
	which, together with \eqref{eq:5.12}, yields
	\begin{equation}\label{eq:5.27}
		\begin{aligned}
			\|\partial_{x}^{3} u\|_{L^{2}}  &\leq \|u_{0}^{\prime \prime \prime}\|_{L^{2}}(1-\delta)^{-\frac{7}{2(1-\delta)^{2}}} q(t)^{-\frac{7}{2(1-\delta)^{2}}}\\
			&\leq C\|u_{0}^{\prime \prime \prime}\|_{L^{2}} q(t)^{-\frac{7}{2(1-\delta)^{2}}}\quad \text{for}\ t \in(0, T_{2}].
		\end{aligned}
	\end{equation}
	
	The next key observation is that one can utilize the Gagliardo-Nirenberg interpolation inequality to deduce that
	\begin{equation}\label{eq:5.28}
		\begin{aligned}
			\|\partial_{x}^{2} u\|_{L^{\infty}}
			&\leq C\|\partial_{x} u\|_{L^{\infty}}^{\frac{1}{3}}\|\partial_{x}^{3} u\|_{L^{2}}^{\frac{2}{3}} \\
			&\leq C  C_{1}^{\frac{1}{3}}\|u_{0}^{\prime \prime \prime}\|_{L^{2}}^{\frac{2}{3}} q(t)^{-\frac{1}{3}-\frac{7}{3(1-\delta)^{2}}}
			\quad \text{for}\ t \in(0, T_{2}],
		\end{aligned}
	\end{equation}
	where one has used \eqref{eq:5.17} and \eqref{eq:5.27}.
	
	It follows from \eqref{eq:5.23} and \eqref{eq:5.28} that
	\begin{equation}\label{eq:5.29}
		|I_{5}| \leq C C_{1}^{\frac{1}{3}}\|u_{0}^{\prime \prime \prime}\|_{L^{2}}^{\frac{2}{3}} \eta^{s} q(t)^{-\frac{1}{3}-\frac{7}{3(1-\delta)^{2}}}.
	\end{equation}
	
	Inserting \eqref{eq:5.24} and \eqref{eq:5.29} to \eqref{K2} and choosing  $\eta=q(t)^{-\frac{2}{3}+\frac{7}{3(1-\delta)^{2}}}$  yield	
	\begin{equation}\label{eq:5.30}
		\begin{aligned}
			|K_{2}(t,x)| & \leq C\Big( C_{1} + C_{1}^{\frac{1}{3}}\|u_{0}^{\prime \prime \prime}\|_{L^{2}}^{\frac{2}{3}}\Big) q(t)^{-\frac{1+2s}{3}-\frac{7(1-s)}{3(1-\delta)^{2}}} \\
			& \leq C\Big(C_{1} + C_{1}^{\frac{1}{3}}\|u_{0}^{\prime \prime \prime}\|_{L^{2}}^{\frac{2}{3}}\Big) q(t)^{-2}\quad \text{for}\ (t,x) \in (0,T_{2}]\times \mathbb{R},
		\end{aligned}
	\end{equation}
	where one has used
	\begin{equation*}
		-\frac{1+2s}{3}-\frac{7(1-s)}{3(1-\delta)^{2}} \geq -2,
	\end{equation*}
	which follows from the facts \(\uwave{s\in (2/5, 1)}\) and \eqref{eq:5.10}.
	
	\underline{\bf{Estimates on \(v_{2}(t,x)\)}.}	By \eqref{eq:4.6},  one notices that
	\begin{equation*}
		\frac{\mathrm{d} v_{2}}{\mathrm{~d} t}\leq|K_{2}(t, x)|,
	\end{equation*}
	which, together with  \eqref{eq:5.30} and \eqref{eq:5.11}, implies	
	\begin{equation}\label{eq:5.31}
		\begin{aligned}
			v_{2}( t,x)
			&\leq \|u_{0}^{\prime}\|_{L^{\infty}} +  C\Big(C_{1} + C_{1}^{\frac{1}{3}}\|u_{0}^{\prime \prime \prime}\|_{L^{2}}^{\frac{2}{3}}\Big)\int_{0}^{t} q(\tau)^{-2} \mathrm{~d} \tau \\
			& \leq \frac{1}{2} C_{1} q(t)^{-1}-C\Big(C_{1} + C_{1}^{\frac{1}{3}}\|u_{0}^{\prime \prime \prime}\|_{L^{2}}^{\frac{2}{3}}\Big)(1-\delta)^{-3} m^{-1}(0) q(t)^{-1} \\
			&< C_{1} q(t)^{-1}\quad \text{for}\ (t,x) \in (0,T_{2}]\times \mathbb{R},
		\end{aligned}
	\end{equation}
	where one has used $\eqref{eq:2.1}_3$ in the last inequality.
	
	On the other hand, one may assume that
	$\|u_{0}^{\prime}\|_{L^{\infty}} = - m(0)$ without loss of generality, which together with \eqref{eq:4.5} and \eqref{eq:2.2} implies
	\begin{equation}\label{eq:5.32}
		\begin{aligned}
			v_{2}(t, x) \geq m(0) q(t)^{-1} \geq -\frac{C_{1}}{2}  q(t)^{-1}\quad \text{for}\ (t,x) \in (0,T_{2}]\times \mathbb{R}.
		\end{aligned}
	\end{equation}
	
	\underline{\bf{The\ contradiction\ argument \eqref{eq:5.15}}}.
	Collecting \eqref{eq:5.22}, \eqref{eq:5.31} and \eqref{eq:5.32}  yields a contradiction to \eqref{eq:5.15}.
	Thus, \eqref{eq:5.13} and \eqref{eq:5.14} follow.
	
	\underline{\bf{The\ contradiction\ argument \eqref{eq:5.5}}}. It follows from \eqref{eq:5.30} and $\eqref{eq:2.1}_1$ that	
	\begin{equation*}
		\begin{aligned}
			|K_{2}(t,x)|
			&\leq C\Big(C_{1} + C_{1}^{\frac{1}{3}}\|u_{0}^{\prime \prime \prime}\|_{L^{2}}^{\frac{2}{3}}\Big) m^{-2}(0) m^{2}(t)\\
			&<\delta^{2} m^{2}(t)\quad
			\text{for}\ (t,x) \in (0,T_{1}]\times \mathbb{R},
		\end{aligned}
	\end{equation*}
	which contradicts \eqref{eq:5.5}. Hence we have shown \eqref{eq:4.6}.\\
	
	Now we are ready to finish the proof of Theorem \ref{th:2.1}.
	\begin{proof}[Proof of Theorem \ref{th:2.1}]	
		For $t \in [0, T )$ and $x \in \Sigma_{\delta}(t)$, it follows from Lemma \ref{lem:3} (by setting $t_{1} = 0$ and $t_{2} = t$) that
		\begin{equation*}
			m(0) \leq v_{2}(0, x) \leq(1-\delta) m(0).
		\end{equation*}
		This, together with \eqref{eq:5.7} and \eqref{eq:5.8}, yields
		\begin{equation*}
			r(t) \leq m(0)\big[v_{2}^{-1}(0,x)+(1-\delta) t\big] \leq (1-\delta)^{-1}+m(0)(1-\delta) t
		\end{equation*}
		and
		\begin{equation*}
			r(t) \geq m(0)\big[v_{2}^{-1}(0,x)+(1+\delta) t\big] \geq 1+m(0)(1+\delta) t.
		\end{equation*}
		Hence
		\begin{equation*}
			(1-\delta)+m(0)(1-\delta^{2}) t \leq q(t) \leq (1-\delta)^{-1}+m(0)(1-\delta) t,
		\end{equation*}
		which is
		\begin{equation}\label{eq:5.33}
			(1-\delta)+\inf _{x \in \mathbb{R}} u_{0}^{\prime}(x)(1-\delta^{2}) t \leq q(t) \leq (1-\delta)^{-1}+\inf _{x \in \mathbb{R}} u_{0}^{\prime}(x)(1-\delta) t.
		\end{equation}
		
		Obviously,  $q(t)$  goes to zero  by sending  $t$ to  $(1+\delta)^{-1}\big[-\inf _{x \in \mathbb{R}} u_{0}^{\prime}(x)\big]^{-1}$   and  $(1-\delta)^{-2}\big[-\inf _{x \in \mathbb{R}} u_{0}^{\prime}(x)\big]^{-1}$  on the LHS and RHS of \eqref{eq:5.33}, respectively. On the other side,  \eqref{eq:5.13} implies that  $v_{1}(t, x)$  is bounded for all  $t \in[0, T^{\prime}]$  with any  $T^{\prime}<T$. Hence, \(u\) exhibits wave breaking at $ T $ satisfying \eqref{eq:2.3}.	
	\end{proof}
	
	
	\section{Proof of Theorem \ref{th:2.2}}
	Note that \(p=1\).
	We only handle the case of \(\uwave{s=1}\) since the other case \(s\in (1, \infty)\) can be dealt with analogously (in fact it is much easier). Comparing the proof in the case of $s \in (2/5, 1)$, our aim in this section is to lower the regularity of \(u\) from \(H^3\) to \(H^2\).

	First we check that \eqref{eq:4.6} holds at $t=0$. $I_{2}$ can be estimated exactly as \eqref{eq:5.3}.
	It remains to consider $I_{1}$.
	By Hölder's inequality, one uses $\eqref{eq:3.1}_2$ and \eqref{eq:3.9} to estimate	
	\begin{equation}\label{eq:6.1}
		\begin{aligned}
			|I_1| &\leq C 	\|u_{0}^{\prime \prime}\|_{L^{2}} \bigg(\int_{|y| < 1} \Big(\log\Big(\frac{1}{|y|}\Big)+1\Big)^{2} \mathrm{~d} y\bigg)^{\frac{1}{2}}\\
			&\leq C \|u_{0}^{\prime \prime}\|_{L^{2}}\bigg(\int_{|y|<1}\bigg(\frac{1}{|y|^{\frac{2}{5}}} +1\bigg)^{2}\mathrm{~d}y\bigg)^{\frac{1}{2}}
			\leq C\|u_{0}\|_{H^{2}}.
		\end{aligned}
	\end{equation}
	It follows from $\eqref{eq:2.4}_1$, \eqref{eq:6.1}  and \eqref{eq:5.3}  that
	\begin{equation*}
		|K_{2}^{s}(0,x)|\leq C(\|u_{0}\|_{H^{2}} + C_{1})  < \delta^{2} m^{2}(0)\quad \text{for}\ x \in \mathbb{R}.
	\end{equation*}

	Now, we shall prove  \eqref{eq:4.6} for $t \ne 0$ by the assumptions $\eqref{eq:2.4}_1$-$\eqref{eq:2.4}_3$. The argument of contradiction, lemmas and claims are exactly the same as in the proof of Theorem \ref{th:2.1}.

	\underline{\bf{Estimates on \(K_{1}^{s}(t, x)\)}.} 	
	One may use Hölder's inequality to estimate	
	\begin{equation}\label{eq:6.3}
		|I_{4}| \leq \|\partial_{x} u\|_{L^{2}}\bigg(\int_{|y|>\eta} G_{s}^{2}(y)  \mathrm{~d} y \bigg)^{\frac{1}{2}} \leq C\|\partial_{x} u\|_{L^{2}},
	\end{equation}
	where one has used \eqref{eq:3.4},	and
	\begin{equation}\label{eq:6.2}
		\begin{aligned}
			|I_{3}|
			&\leq C\|\partial_{x} u\|_{L^{2}} \bigg(\int_{|y| \leq \eta} \Big(\log\Big(\frac{1}{|y|}\Big)+1\Big)^{2} \mathrm{~d} y\bigg)^{\frac{1}{2}}\\
			&\leq C 	\|\partial_{x} u\|_{L^{2}}\bigg(\int_{|y| \leq \eta}\bigg(\frac{1}{|y|^{\frac{2}{5}}}+1\bigg)^{2} \mathrm{~d} y\bigg)^{\frac{1}{2}}
			\leq C  \|\partial_{x} u\|_{L^{2}}\eta^{\frac{1}{10}}
		\end{aligned}
	\end{equation}
	where one has used $\eqref{eq:3.1}_{2}$ and \eqref{eq:3.9}.

	In order to get the blow-up rate of \(I_{3}\) and \(I_{4}\), it suffices to estimate  $\|\partial_{x}u\|_{L^2}$, which will be achieved by energy estimates. Applying $\partial_{x}$ to \eqref{eq:1.1}, and multiplying it by $\partial_{x}u$, one finds that
	\begin{equation*}
		\begin{aligned}
			&\frac{\mathrm{d}}{\mathrm{d} t} \int_{\mathbb{R}}(\partial_{x} u)^{2} \mathrm{~d} x =-\int_{\mathbb{R}} \partial_{x} u(\partial_{x} u)^{2} \mathrm{~d} x\\
			& \leq -m(0) q(t) ^{-1}\int_{\mathbb{R}}(\partial_{x} u)^{2} \mathrm{~d} x\leq -m(0) q(t)^{-1}\|\partial_{x} u\|^{2}_{L^2},
		\end{aligned}
	\end{equation*}
	which, together with \eqref{eq:5.12}, gives
	\begin{equation}\label{eq:6.4}
		\begin{aligned}
			\|\partial_{x} u(t)\|_{L^{2}}
			&\leq\|u_{0}^{\prime}\|_{L^{2}}(1-\delta)^{-\frac{1}{2(1-\delta)^{2}}} q(t)^{-\frac{1}{2(1-\delta)^{2}}}\\
			&\leq C\|u_{0}^{\prime}\|_{L^{2}} q(t)^{-\frac{1}{2(1-\delta)^{2}}}\quad \text{for}\ t \in(0, T_{2}].
		\end{aligned}
	\end{equation}

	It follows from \eqref{eq:6.2}, \eqref{eq:6.3} and \eqref{eq:6.4} that
	\begin{equation}\label{eq:6.5}
		|I_{3}| \leq C \|u_{0}^{\prime}\|_{L^{2}}\eta^{\frac{1}{10}} q(t)^{-\frac{1}{2(1-\delta)^{2}}}
	\end{equation}
	and
	\begin{equation}\label{eq:6.6}
		|I_{4}| \leq C\|u_{0}^{\prime}\|_{L^{2}} q(t)^{-\frac{1}{2(1-\delta)^{2}}}.
	\end{equation}
	
	Choosing  $\eta=1$, we conclude from \eqref{eq:6.5} and \eqref{eq:6.6}  that	
	\begin{equation}\label{eq:6.7}
		\begin{aligned}
			|K_{1}^{s}(t,x)|& \leq C\|u_{0}^{\prime}\|_{L^{2}} q(t)^{-\frac{1}{2(1-\delta)^{2}}}\\
			&\leq C\|u_{0}^{\prime}\|_{L^{2}} q(t)^{-\frac{2}{3}}\quad \text{for}\ (t,x) \in (0,T_{2}]\times \mathbb{R},
		\end{aligned}
	\end{equation}
	where one has used
	\begin{equation*}
		-\frac{1}{2(1-\delta)^{2}} \geq -\frac{2}{3}.
	\end{equation*}

	\underline{\bf{Estimates on \(v_{1}(t,x)\)}.}
	By \eqref{eq:2.2}, \eqref{eq:6.7} and \eqref{eq:5.11}, one may estimate
	\begin{equation}\label{eq:6.8}
		\begin{aligned}
			|v_{1}(t,x)| & \leq\|u_{0}\|_{L^{\infty}}+\int_{0}^{t}|K_{1}^{s}(\tau, x)| \mathrm{~d} \tau \\
			& \leq \frac{1}{2} C_{0}+C\|u_{0}^{\prime}\|_{L^{2}} \int_{0}^{t} q^{-\frac{2}{3}}(\tau) \mathrm{~d} \tau \\
			& \leq  \frac{1}{2} C_{0}-C\|u_{0}^{\prime}\|_{L^{2}}(1-\delta)^{-2} m^{-1}(0)  \\
			& <C_{0}\quad \text{for}\ (t,x) \in (0,T_{2}]\times \mathbb{R},
		\end{aligned}
	\end{equation}
	where  $\eqref{eq:2.4}_2$ has been used in the last inequality.
	
	\underline{\bf{Estimates on \(K_{2}^{s}(t, x)\)}.} For \(I_{6}\), similar to \eqref{eq:5.24},
	one uses \eqref{eq:3.1}$_{2}$, \eqref{eq:3.2}$_{1}$ and \eqref{eq:3.3} to get
	\begin{equation}\label{eq:6.9}
		\begin{aligned}
			|I_{6}| &\leq  |G_{s}(\eta)[\partial_{x} u(t, X(t, x)-\eta)-\partial_{x} u(t, X(t, x)+\eta)]| \\
			&\quad +\bigg|\int_{\eta<|y| \leq 1} G_{s}^{\prime}(y) \partial_{x} u(t, X(t, x)-y) \mathrm{~d} y\bigg| \\
			&\quad +\bigg|\int_{|y|>1} G_{s}^{\prime}(y) \partial_{x} u(t, X(t, x)-y) \mathrm{~d} y\bigg| \\
			&\leq C\bigg[\bigg(\log \bigg(\frac{1}{\eta}\bigg)+1\bigg)\|v_{2}\|_{L^{\infty}}+|\log \eta| \|v_{2}\|_{L^{\infty}}+\|v_{2}\|_{L^{\infty}}\bigg]\\
			&\leq  C  C_{1}\eta^{-\frac{1}{6}}  q(t)^{-1},
		\end{aligned}
	\end{equation}
	where one has used
	\begin{equation*}
		\log\bigg(\frac{1}{\eta}\bigg)\leq \frac{C}{\eta^{\frac{1}{6}}}\quad \text{for} \ 0<\eta<1.
	\end{equation*}
	For \(I_{5}\), one instead uses Hölder's inequality to estimate
	\begin{equation*}
		|I_{5}|\leq C  \|\partial_{x}^{2} u\|_{L^{2}}\eta^{\frac{1}{10}}.	
	\end{equation*}	
	
	It remains to control $\|\partial_{x}^{2} u\|_{L^{2}}$.  Indeed, applying $\partial^{2}_{x}$ to \eqref{eq:1.1} and multiplying it by $\partial^{2}_{x}u$, one deduces that
	\begin{equation*}
		\frac{\mathrm{d}}{\mathrm{d} t} \int_{\mathbb{R}}(\partial^{2}_{x} u)^{2} \mathrm{~d} x=- 5\int_{\mathbb{R}}\partial_{x} u(\partial^{2}_{x} u)^{2} \mathrm{~d} x \leq- 5m(0) q^{-1}(t)\|\partial^{2}_{x} u\|_{L^{2}}^{2},
	\end{equation*}
	which, along with \eqref{eq:5.12}, yields 	
	\begin{equation*}
		\begin{aligned}
			\|\partial_{x}^{2} u\|_{L^{2}}  &\leq \|u_{0}^{\prime\prime }\|_{L^{2}}(1-\delta)^{-\frac{5}{2(1-\delta)^{2}}} q(t)^{-\frac{5}{2(1-\delta)^{2}}}\\
			&\leq C \|u_{0}^{\prime\prime}\|_{L^{2}} q(t)^{-\frac{5}{2(1-\delta)^{2}}}\quad \text{for}\ t \in(0, T_{2}].
		\end{aligned}
	\end{equation*}
	Hence, one obtains
	\begin{equation}\label{eq:6.10}
		|I_{5}|\leq C\| u_{0}^{\prime \prime}\|_{L^{2}}\eta^{\frac{1}{10}} q(t)^{-\frac{5}{2(1-\delta)^2}}.
	\end{equation}

	Taking $\eta=q(t)^{-\frac{15}{4}+\frac{75}{8(1-\delta)^{2}}}$, it follows from \eqref{K2}, \eqref{eq:6.9} and \eqref{eq:6.10} that
	\begin{equation}\label{eq:6.11}
		\begin{aligned}
			|K^{s}_{2}(t,x)|
			& \leq C\big( C_{1}+ \|u_{0}^{\prime \prime}\|_{L^{2}}\big) q(t)^{-\frac{3}{8}-\frac{25}{16(1-\delta)^{2}}} \\
			& \leq C\big(C_{1} + \|u_{0}^{\prime \prime}\|_{L^{2}}\big)q(t)^{-2}\quad \text{for}\ (t,x) \in (0,T_{2}]\times \mathbb{R},
		\end{aligned}
	\end{equation}	
	where one has used
	\begin{equation*}
		-\frac{3}{8}-\frac{25}{16(1-\delta)^{2}} \geq -2.
	\end{equation*}

	\underline{\bf{Estimates on \(v_{2}(t,x)\)}.}
	It follows from \eqref{eq:6.11} and \eqref{eq:5.11} that
	\begin{equation}\label{eq:6.12}
		\begin{aligned}
			v_{2}(t, x) & \leq \|u_{0}^{\prime}\|_{L^{\infty}}+C\big(C_{1} + \|u_{0}^{\prime \prime}\|_{L^{2}}\big)  \int_{0}^{t} q(\tau)^{-2} \mathrm{~d} \tau \\
			& \leq \frac{1}{2} C_{1} q^{-1}(t)-C(1-\delta)^{-3} m^{-1}(0)\big(C_{1} +  \|u_{0}^{\prime \prime}\|_{L^{2}}\big) q(t)^{-1} \\
			& <C_{1} q(t)^{-1}\quad \text{for}\ (t,x) \in (0,T_{2}]\times \mathbb{R},
		\end{aligned}
	\end{equation}	
	where $\eqref{eq:2.4}_3$ has been used. Moreover, similar to \eqref{eq:5.32}, it holds that	
	\begin{equation}\label{eq:6.13}
		v_{2}(t, x) \geq -\frac{C_{1}}{2} q(t)^{-1}\quad \text{for}\ (t,x) \in (0,T_{2}]\times \mathbb{R}.
	\end{equation}

	\underline{\bf{The\ contradiction\ argument \eqref{eq:5.15}}}. By \eqref{eq:6.8}, \eqref{eq:6.12} and \eqref{eq:6.13},
	we get a contradiction to \eqref{eq:5.15}.
	Hence, \eqref{eq:5.13} and \eqref{eq:5.14} follow.
	
	\underline{\bf{The\ contradiction\ argument \eqref{eq:5.5}}}. One uses \eqref{eq:6.11} and $\eqref{eq:2.4}_1$ to find that	
	\begin{equation*}
		\begin{aligned}
			|K_{2}(t,x)| &\leq C \big(C_{1} + \|u_{0}^{\prime \prime}\|_{L^{2}}\big)  m^{-2}(0) m^{2}(t)\\
			&<\delta^{2} m^{2}(t)\quad \text{for}\ (t,x) \in (0,T_{1}]\times \mathbb{R},
		\end{aligned}
	\end{equation*}	
	which contradicts \eqref{eq:5.5}. This completes the proof of \eqref{eq:4.6}.
	
	The remaining proof is similar to that of  $2/5<s<1$, so we omit it.


	\section{Proof of Theorem \ref{th:2.3}}
	
	Note that \(p>1\).  We will point only the main differences in the proof between Theorem \ref{th:2.1} and \ref{th:2.3}.

	First, one can check \eqref{eq:5.4} by following exactly the same way as \eqref{eq:5.1}-\eqref{eq:5.3} and using $\eqref{eq:2.5}_1$. Then, it remains to show \eqref{eq:4.6} by an argument of contradiction.

	Since the nonlinear term  $v_{1}^{p-1}v_{2}^2$  does not have a fixed sign generally,  in order to use  $v_{1}^{p-1}v_{2}^2$  to control  $K_{2}(t,  x)$  in \eqref{eq:4.3},  the key idea is to make the following a priori assumption:
	\begin{equation}\label{eq:7.1}
		\boxed{	A \leq v_{1}(t;x) \leq B  \quad \text { for } (t,x) \in [0,  T_{1}]\times [\bar{x}_{1},  \bar{x}_{2}]},
	\end{equation}
	in which  $A$  and  $B$  are given in Theorem
	\ref{th:2.3} satisfying \eqref{eq:2.6}.
	
	For \((t,x) \in [0,  T_{1}]\times [\bar{x}_{1},  \bar{x}_{2}]\), define
	\begin{equation*}
		\Sigma_{\delta}(t)=\{x\in[\bar{x}_{1}, \bar{x}_{2}]: v_{2}(t; x) \leq (A^{p-1}B^{1-p} - \delta)m(t)\}
	\end{equation*}
	and
	\begin{equation*}
		v_{2}(t, x)\bigg(=\frac{v_{2}(0)}{1+v_{2}(0) \int_{0}^{t}\big[pv_{1}^{p-1}(\tau)+\big(v_{2}^{-2} K_{1}\big)(\tau)\big]\, \, \diff \tau}\bigg)=\colon m(0) r^{-1}(t, x).
	\end{equation*}

	Then, one can show the following lemmas.
	
	\begin{lemma}~\label{lem:6}
		For fixed \(\delta\), the set $\Sigma_{\delta}(t)$ is decreasing in \(t\), namely
		$\Sigma_{\delta}(t_{2}) \subset \Sigma_{\delta}(t_{1})$  whenever  $0 \leq t_{1} \leq t_{2} \leq T_{1} $.
	\end{lemma}
	
	\begin{lemma}\label{lem:7}
		It holds that
		\begin{equation*}
			(pB^{p-1}+\delta) m(0) \leq\frac{\diff}{\diff t}r(t,x) \leq (pA^{p-1}-\delta) m(0)<0 \quad \mathrm{for}\ x \in\Sigma_{\delta}(T_1),
		\end{equation*}
		\begin{equation*}
			q(t) \leq r(t,x) \leq \frac{1}{A^{p-1}B^{1-p}-\delta} q(t)  \quad \mathrm{for}\ x \in\Sigma_{\delta}(T_1)
		\end{equation*}
		and
		\begin{equation*}
			0<q(t) \leq 1.
		\end{equation*}
	\end{lemma}
	
	\begin{lemma}\label{lem:8}
		It holds that
		\begin{equation}\label{eq:7.2}
			\begin{aligned}
				\int_0^t q^{-\gamma}(\tau)\,\mathrm d\tau &\leq(1-\gamma)^{-1}m^{-1}(0)(pA^{p-1}-\delta)^{-1}(A^{p-1}B^{1-p}-\delta)^{-\gamma}\\
				&\quad\times\big[q^{1-\gamma}(t)-(A^{p-1}B^{1-p}-\delta)^{\gamma-1}\big],
			\end{aligned}
		\end{equation}		
		where  $\gamma \in (0, 1) \cup (1, \infty)$, and		
		\begin{equation}\label{eq:7.3}
			\begin{aligned}
				\int_0^t q^{-1}(\tau)\,\mathrm d\tau &\leq m^{-1}(0)(pA^{p-1}-\delta)^{-1}(A^{p-1}B^{1-p}-\delta)^{-1}\\
				&\quad\times\big[\log(A^{p-1}B^{1-p}-\delta)+\log q(t)\big].
			\end{aligned}
		\end{equation}		
	\end{lemma}

	\underline{\bf{Estimates on \(K_{1}^{s}(t, x)\)}.}  The estimates are exactly same as in the proof of Theorem \ref{th:2.1}.
	
	\underline{\bf{Estimates on \(v_{1}(t,x)\)}.}
	Note that
	\begin{equation*}
		\begin{aligned}
			(pA^{p-1}-\delta)^{-1}(A^{p-1}B^{1-p}- \delta)^{-1}<(1-\delta)^{-2}.
		\end{aligned}
	\end{equation*}	
	Then it follows from \eqref{eq:4.1} and \eqref{eq:7.2} that
	\begin{equation}\label{eq:7.4}
		\begin{aligned}
			|v_{1}(t;x)| & \leq\|u_{0}\|_{L^{\infty}}+\int_{0}^{t}|K_{1}^{s}(\tau,  x)| \mathrm{~d} \tau \\
			& \leq \frac{1}{2} C_{0}- C(C_{0} +  C_{1})\big(pA^{p-1}-\delta\big)^{-1}\big(A^{p-1}B^{1-p}- \delta\big)^{s-1} \\
			&\quad\times m^{-1}(0) \big[(A^{p-1}B^{1-p}-\delta)^{-s} - q^{s}(t)\big]\\
			& \leq \frac{1}{2} C_{0}-C(C_{0} + C_{1})(1-\delta)^{-2} m^{-1}(0) \\
			& <C_{0} \quad \text{for }~(t,x) \in (0, T_{2}]\times \mathbb{R},
		\end{aligned}
	\end{equation}
	where $\eqref{eq:2.5}_2$ has been used in the last inequality.
	
	\underline{\bf{Estimates on \(K_{2}^{s}(t, x)\)}.} By \eqref{eq:7.3}, one solves \eqref{eq:5.26} to find that
	\begin{equation*}
		\begin{aligned}
			\|\partial_{x}^{3} u\|_{L^{2}}  &\leq\|u_{0}^{\prime \prime \prime}\|_{L^{2}}\big(A^{p-1}B^{1-p}-\delta\big)^{-\frac{7}{2(pA^{p-1}-\delta)(A^{p-1}B^{1-p}-\delta)}}\\
			&\quad \times q(t)^{-\frac{7}{2(pA^{p-1}-\delta)(A^{p-1}B^{1-p}-\delta)}}\\
			&\leq\bigg(\frac{A^{p-1}}{2B^{p-1}}\bigg)^{-\frac{7B^{p-1}}{2pA^{2p-2}}}\|u_{0}^{\prime \prime \prime}\|_{L^{2}} q(t)^{-\frac{7 B^{p-1}}{2p A^{2p-2}}}\quad\text{for } ~ t \in(0,  T_{2}],
		\end{aligned}			
	\end{equation*}
	which, along with \eqref{eq:5.23}, yields
	\begin{equation*}
		|I_{5}| \leq  C \bigg(\frac{A^{p-1}}{2B^{p-1}} \bigg)^{-\frac{7B^{p-1}}{2pA^{2p-2}}}\|u_{0}^{\prime \prime \prime}\|_{L^{2}}   \eta^{s}q(t)^{-\frac{7 B^{p-1}}{2p A^{2p-2}}}.
	\end{equation*}
	This together with \eqref{K2}, \eqref{eq:5.24} by choosing  $\eta=q(t)^{-1+\frac{7B^{p-1}}{2p A^{2p-2}}}$ gives
	{\small\begin{equation*}\label{eq:7.9}
			\begin{aligned}
				|K^{s}_{2}(t;x)| &\leq  C\bigg[C_{1}  +\bigg(\frac{A^{p-1}}{2B^{p-1}}\bigg)^{-\frac{7B^{p-1}}{2pA^{2p-2}}}\|u_{0}^{\prime \prime \prime}\|_{L^{2}} \bigg]  q(t)^{\big(1-\frac{7B^{p-1}}{p A^{2p-2}}\big)(1-s)-1} \\
				&\leq C\bigg[C_{1}  +\bigg(\frac{A^{p-1}}{2B^{p-1}}\bigg)^{-\frac{7B^{p-1}}{2pA^{2p-2}}}\|u_{0}^{\prime \prime \prime}\|_{L^{2}} \bigg] q(t)^{-2}
				\quad\text{for } ~ (t,x) \in(0,  T_{2}] \times\mathbb{R},
			\end{aligned}
	\end{equation*}}

\noindent	where one has used
	\begin{equation*}
		\bigg(1-\frac{7B^{p-1}}{2pA^{2p-2}}\bigg)(1-s)-1\geq -2,
	\end{equation*}
	which follows from  \eqref{eq:2.7}, $\uwave{s\in (0, 1)}$, and $0<q(t)\leq 1$.
	
	\underline{\bf{Estimates on \(v_{2}(t,x)\)}.}	
	Note that
	\begin{equation*}
		\begin{aligned}
			(pA^{p-1}-\delta)^{-1}(A^{p-1}B^{1-p}- \delta)^{-2}<(1-\delta)^{-3}.
		\end{aligned}
	\end{equation*}	
	Then	it follows from \eqref{eq:4.3} and \eqref{eq:7.3} that
	{\small	\begin{equation*}
			\begin{aligned}
				v_{2}( t;x)
				&\leq C\bigg[C_{1}  +\ \bigg(\frac{A^{p-1}}{2B^{p-1}}\bigg)^{-\frac{7B^{p-1}}{2pA^{2p-2}}}\|u_{0}^{\prime \prime \prime}\|_{L^{2}} \bigg] \int_{0}^{t} q(\tau)^{-2} \mathrm{~d} \tau \\
				&\leq \frac{1}{2} C_{1} q(t)^{-1}-C\bigg[C_{1}  +\ \bigg(\frac{A^{p-1}}{2B^{p-1}}\bigg)^{-\frac{7B^{p-1}}{2pA^{2p-2}}}\|u_{0}^{\prime \prime \prime}\|_{L^{2}} \bigg](pA^{p-1}-\delta)^{-1} \\
				&\quad\times(A^{p-1}B^{1-p}-\delta)^{-2}m^{-1}(0)\big[q(t)^{-1}-(A^{p-1}B^{1-p}-\delta)\big] \\
				&\leq \frac{1}{2} C_{1} q(t)^{-1}-C\bigg[C_{1}  +\bigg(\frac{A^{p-1}}{2B^{p-1}}\bigg)^{-\frac{7B^{p-1}}{2pA^{2p-2}}}\|u_{0}^{\prime \prime \prime}\|_{L^{2}} \bigg]\\
				&\quad\times  (1-\delta)^{-3} m^{-1}(0) q^{-1}(t) \\
				&< C_{1} q(t)^{-1} \quad \text{for }~(t,x) \in (0, T_{2}]\times \mathbb{R},
			\end{aligned}
	\end{equation*}}
	
	\noindent where $\eqref{eq:2.5}_3$ has been used in the last inequality.

	\underline{\bf{The\ contradiction\ arguments \eqref{eq:5.15} and \eqref{eq:5.5}}}. One can get contradictions to \eqref{eq:5.15} and \eqref{eq:5.5} by collecting the estimates above. 
	
	\underline{\bf{The\ a\ priori\ assumption \eqref{eq:7.1}}}. Similar to \eqref{eq:7.4},  one may estimate
	\begin{equation}\label{eq:7.5}
		\begin{aligned}
			&v_{1}(t,  x) \leq u_{0}(x)-C m^{-1}(0)(1-\delta)^{-2}(C_{0} + C_{1})<B,\\
			&v_{1}(t,  x) \geq  u_{0}(x)+ C m^{-1}(0)(1-\delta)^{-2}(C_{0} + C_{1})>A
		\end{aligned}
	\end{equation}
	for $t \in[0,  T_{1}]$  and  $x \in[\bar{x}_{1},  \bar{x}_{2}]$.
	Here one has used $\eqref{eq:2.6}_1$ in \(\eqref{eq:7.5}_1\) and $\eqref{eq:2.6}_2$ in \(\eqref{eq:7.5}_2\), respectively.

	\section{Proof of Theorem \ref{th:2.4}}	
	
	One can follow the arguments in showing Theorem \ref{th:2.2} and Theorem \ref{th:2.3}	to prove Theorem \ref{th:2.4}.
	
	\section{Dispersive properties and weak entropy solutions}
	
	\subsection{Linear estimates}

	We have seen that \eqref{eq:1.1} shares with the Burgers equation a typical property of conservation laws, the possibility of shock formation.
	We briefly comment here on dispersive properties. The first one concerns $L^1-L^\infty$ estimates for the linear equation
	\begin{equation}\label{lin}
\begin{aligned}
	u_t+(I-\partial_x^2)^{-s/2}u_x=0, \quad u(x,0)=\phi(x).
\end{aligned}
	\end{equation}
	
	The case $s=2$ (linearized Fornberg-Whitham equation) corresponds to the linear Benjamin-Bona-Mahony (BBM ) equation	
	\begin{equation}\label{lBBM}
\begin{aligned}
	u_t+(I-\partial_x^2)^{-1}u_x=0,
\end{aligned}
	\end{equation}
for which J. Albert in \cite{MR1009061} proved the following decay estimate for the solution $u$ of \eqref{lBBM} with initial data $\phi\in L^1(\R)\cap H^4(\R)$:
\begin{equation*}
\begin{aligned}	
	\|u(\cdot,t)\|_{L^\infty}\lesssim (\|\phi\|_{L^1}+\|\phi\|_{H^4})(1+t)^{-1/3},\quad \forall t\geq 1.
\end{aligned}
\end{equation*}
In \cite{MR840595}, Albert proved a similar decay estimate in a different functional setting, that is 
	\begin{equation*}
\begin{aligned}
	\|u(\cdot,t)\|_{L^\infty}\lesssim \|(1+|x|)\phi\|_{L^2}(1+t)^{-1/3},\quad \forall t\geq 1.
\end{aligned}
\end{equation*}		
Similar linear estimates hold for \eqref{lin} as well. 

Let 
\begin{equation*}
\begin{aligned}
\Phi(\xi)=\Phi(\xi;x,t):=t^{-1}x\xi+(1+|\xi|^2)^{s/2}\xi.
\end{aligned}
\end{equation*}	
One may write
\begin{equation*}
\begin{aligned}
e^{t(1-\partial_x^2)^{s/2}\partial_x}  g(t,x)
&=\frac{1}{\sqrt{2\pi}}\int_{-\infty}^\infty e^{\mathrm{i}t\Phi(\xi)}\widehat{g}(t,\xi)\varphi(2^{10}\xi)\,\diff \xi\\
&\quad+\frac{1}{\sqrt{2\pi}}\int_{-\infty}^\infty e^{\mathrm{i}t\Phi(\xi)}\widehat{g}(t,\xi)\big(1-\varphi(2^{10}\xi)\big)\,\diff \xi\\
&=\colon I_{L}(t,x,g)+I_{H}(t,x,g).
\end{aligned}
\end{equation*}
Then, we have the following decay estimates. 
\begin{proposition}\label{decay:linear} Let \(N_0=[s]+1\), \(t\geq 1,\ x\in\R\) and \(g\) be a real function. Assume
\begin{equation*}
	\begin{aligned}
	\|\widehat{g}\|_{L^\infty}+t^{-1/6}(\|g\|_{H^{1,1}}+\|g\|_{H^{N_0}})\leq 1.
	\end{aligned}
	\end{equation*}
Then
	\begin{equation}\label{7}
	\begin{aligned}
	|I_{L}(t,x,g)|\lesssim t^{-1/3}\langle (x+t)/t^{1/3}\rangle^{-1/4}
	\end{aligned}
	\end{equation}
and	
	\begin{equation}\label{8}
	\begin{aligned}
	|I_{H}(t,x,g)|\lesssim t^{-1/3}. 
	\end{aligned}
	\end{equation}	

Consequently, for the solution \(u\) to \eqref{eq:1.1}, it holds that
	\begin{equation*}
	\begin{aligned}
	\|u(\cdot,t)\|_{L^\infty}\lesssim t^{-1/3}[\|u\|_{L^1}+t^{-1/6}(\|u\|_{H^{1,1}}+\|u\|_{H^{N_0}})],\quad \forall t\geq 1.
	\end{aligned}
	\end{equation*}
	
\end{proposition}

\begin{proof}[Proof of \eqref{7}]  The proof is close to \cite{MR3519470} (see also \cite{MR4278400}). 	
	It is easy to see that 
	\(\partial_\xi\Phi(\xi)=0\) on \([-2^{-9},2^{-9}]\) has no root or two roots with opposite signs (corresponding to \(x>-t\)). It suffices to consider the latter case since the former case is much easier and follows from similar calculations. 
	Let \(\xi_0\) be the positive root of \(\partial_\xi\Phi(\xi_0)=0\) on \([-2^{-9},2^{-9}]\), and set
\begin{align*}
	I_L^{+}:=\int_0^\infty e^{\mathrm{i}t\Phi(\xi)}\widehat{g}(t,\xi)\varphi(2^{10}\xi)\,\diff \xi.
\end{align*}
To verify \eqref{7}, up to taking complex conjugates, it suffices to show 
	\begin{align*}
	|I_L^{+}|\lesssim t^{-1/3}\max(t^{1/3}\xi_0,1)^{-1/2},
	\end{align*}
which will be divided into two cases depending on the size of \(\xi_0\). 
	
	{\textbf{Case 1: \(\xi_0\leq t^{-1/3}\)}.} In this case, one needs only to show that $I_L^{+}$ is bounded by \(t^{-1/3}\).
For this, we decompose \(I_L^{+}\) as follows: 
	\begin{equation*}
	\begin{aligned}
	I_L^{+}
	&=\int_0^\infty e^{\mathrm{i}t\Phi(\xi)}\widehat{g}(t,\xi)\varphi(2^{-10}t^{1/3}\xi)\varphi(2^{10}\xi)\,\diff \xi\\
	&\quad+\int_0^\infty e^{\mathrm{i}t\Phi(\xi)}\widehat{g}(t,\xi)\big(1-\varphi(2^{-10}t^{1/3}\xi)\big)\varphi(2^{10}\xi)\,\diff \xi
=:A_{1}+A_{2}.
	\end{aligned}
	\end{equation*}
Clearly \(A_1\) can be controlled by the desired bound \(t^{-1/3}\).
To estimate \(A_2\), one uses integration by part to find that
	\begin{equation*}
	\begin{aligned}
	|A_2|&\lesssim 
	t^{-1}\int_0^\infty \big|\partial_\xi\big[(\partial_\xi\Phi)^{-1}\big(1-\varphi(2^{-10}t^{1/3}\xi)\big)\varphi(2^{10}\xi)\big]\widehat{g}(t,\xi)\big|\,\diff \xi\\
	&\quad+t^{-1}\int_0^\infty \big|(\partial_\xi\Phi)^{-1}\big(1-\varphi(2^{-10}t^{1/3}\xi)\big)\varphi(2^{10}\xi)\partial_\xi\widehat{g}(t,\xi)\big|\,\diff \xi\\
&=:A_{21}+A_{22}
	\end{aligned}
	\end{equation*}
with
\begin{equation*}
	\begin{aligned}
	A_{21}&\lesssim 
	t^{-1}\int_0^\infty \big|\partial_\xi\big[(\partial_\xi\Phi)^{-1}\big(1-\varphi(2^{-10}t^{1/3}\xi)\big)]\varphi(2^{10}\xi)\widehat{g}(t,\xi)\big|\,\diff \xi\\
	&\quad+t^{-1}\int_0^\infty \big|(\partial_\xi\Phi)^{-1}\big(1-\varphi(2^{-10}t^{1/3}\xi)\partial_\xi\varphi(2^{10}\xi)\widehat{g}(t,\xi)\big|\,\diff \xi\\
&=:A_{21}^1+A_{21}^2.
	\end{aligned}
	\end{equation*}
	The key point in estimating \(A_{21}\) and \(A_{22}\) is to bound \(|\partial_\xi\Phi|\) from below. 
	In fact, by Taylor's formula, one has 
	\begin{equation*}
	\begin{aligned}
	|\partial_\xi\Phi|=\bigg|\frac{3s}{2}(\xi_0^2-\xi^2)+\mathcal{O}(\xi_0^4)+\mathcal{O}(\xi^4)\bigg|
	\gtrsim |\xi_0^2-\xi^2|\gtrsim \xi^2
	\end{aligned}
	\end{equation*}
since \(\xi\gg \xi_0\) on the support of \(A_2\). Since \(\varphi(2^{10}\xi)\) is bounded, \(A_{21}^1\) and \(A_{22}\) can be shown similarly as \cite{MR3519470} to have the bound \(t^{-1/3}\). For the new term \(A_{21}^2\), one has
\begin{equation*}
\begin{aligned}
A_{21}^2&\lesssim t^{-1}\int_0^\infty \big|(\partial_\xi\Phi)^{-1}\big(1-\varphi(2^{-10}t^{1/3}\xi)\big)\partial_\xi\varphi(2^{10}\xi)\widehat{g}(t,\xi)\big|\,\diff \xi\\
&\lesssim t^{-1}\|\widehat{g}\|_{L^\infty}\int_0^\infty \xi^{-2}|1-\varphi(2^{-10}t^{1/3}\xi)|\,\diff \xi
\lesssim t^{-2/3}.
\end{aligned}
\end{equation*}
Hence \(A_2\) is bounded by the desired bound \(t^{-1/3}\).

{\textbf{Case 2: \(\xi_0\geq t^{-1/3}\)}.} In this case, one shall obtain a bound \(t^{-1/2}\xi_0^{-1/2}\) for $I_L^{+}$. 
	To this end, one instead decomposes
	\begin{equation*}
	\begin{aligned}
	I_L^{+}
	&=\int_0^\infty e^{\mathrm{i}t\Phi(\xi)}\widehat{g}(t,\xi)\big(1-\psi(\xi/\xi_0)\big)\varphi(2^{10}\xi)\,\diff \xi\\
	&\quad+\int_0^\infty e^{\mathrm{i}t\Phi(\xi)}\widehat{g}(t,\xi)\psi(\xi/\xi_0)\varphi(2^{10}\xi)\,\diff \xi
=:A_3+A_4.
	\end{aligned}
	\end{equation*}
Integration by parts leads to
	\begin{equation*}
	\begin{aligned}
	A_3&\lesssim t^{-1}\int_0^\infty \left|\partial_\xi\big[(\partial_\xi\Phi)^{-1}\big(1-\psi(\xi/\xi_0)\big)\varphi(2^{10}\xi)\big]\widehat{g}(t,\xi)\right|\,\diff \xi\\
	&\quad+t^{-1}\int_0^\infty \left|(\partial_\xi\Phi)^{-1}\big(1-\psi(\xi/\xi_0)\big)\varphi(2^{10}\xi)\partial_\xi\widehat{g}(t,\xi)\right|\,\diff \xi\\
&=:A_{31}+A_{32}
	\end{aligned}
	\end{equation*}
with
\begin{equation*}
	\begin{aligned}
	A_{31}&\lesssim t^{-1}\int_0^\infty \left|\partial_\xi\big[(\partial_\xi\Phi)^{-1}\big(1-\psi(\xi/\xi_0)\big)\big]\varphi(2^{10}\xi)\widehat{g}(t,\xi)\right|\,\diff \xi\\
	&\quad+t^{-1}\int_0^\infty \left|(\partial_\xi\Phi)^{-1}\big(1-\psi(\xi/\xi_0)\big)\partial_\xi\varphi(2^{10}\xi)\widehat{g}(t,\xi)\right|\,\diff \xi\\
&=:A_{31}^1+A_{32}^2.
	\end{aligned}
	\end{equation*}
Note that
	$
	|\partial_\xi\Phi|
	\gtrsim |\xi_0^2-\xi^2|\gtrsim \max(\xi,\xi_0)^2
	$
on the support of \(A_3\). Again, since \(\varphi(2^{10}\xi)\) is bounded,  one can follow \cite{MR3519470} to show \(A_{31}^1\) and \(A_{32}\) to have the bound \(t^{-1/2}\xi_0^{-1/2}\). It remains to consider \(A_{32}^2\). Indeed, it holds that
	\begin{equation*}
	\begin{aligned}
	A_{32}^2
	\lesssim t^{-1}\|\widehat{g}\|_{L^\infty}\int_0^\infty \max(\xi,\xi_0)^{-2}|1-\psi(\xi/\xi_0)|\,\diff \xi
	\lesssim t^{-1}\xi_0^{-3},
	\end{aligned}
	\end{equation*}
which is better than the desired bound
	\(t^{-1/2}\xi_0^{-1/2}\) due to \(\xi_0\geq t^{-1/3}\). Thus, we have shown that \(A_3\) satisfies the desired bound \(t^{-1/2}\xi_0^{-1/2}\).

We next deal with \(A_4\).
  Let \(l_0\) be the smallest integer satisfying \(2^{l_0}\geq (t\xi_0)^{-1/2}\). Then, it holds that
	\begin{equation*}
	\begin{aligned}
	A_4&= \int_{-\infty}^\infty e^{\mathrm{i}t\Phi(\xi)}\widehat{g}(t,\xi)\psi(\xi/\xi_0)\varphi\big(2^{-l_0}(\xi-\xi_0)\big)\varphi(2^{10}\xi)\,\diff \xi\\
	&\quad+\sum_{l\geq l_0+1}\int_{-\infty}^\infty e^{\mathrm{i}t\Phi(\xi)}\widehat{g}(t,\xi)\psi(\xi/\xi_0)\psi\big(2^{-l}(\xi-\xi_0)\big)\varphi(2^{10}\xi)\,\diff \xi\\
&=:A_{4l_0}+\sum_{l\geq l_0+1}A_{4l}.
	\end{aligned}
	\end{equation*}
First, one has
	\begin{equation*}
	\begin{aligned}
	|A_{4l_0}|
	\lesssim 2^{l_0}\|\widehat{g}\|_{L^\infty}\lesssim t^{-1/2}\xi_0^{-1/2}.
	\end{aligned}
	\end{equation*}
It remains to estimate \(A_{4l}\) for \(l\geq l_0+1\). For this, one integrates by parts to deduce that
	\begin{equation*}
	\begin{aligned}
	|A_{4l}|&\lesssim 
	t^{-1}\int_{-\infty}^\infty \big|\partial_\xi\big[(\partial_\xi\Phi)^{-1}\psi(\xi/\xi_0)\psi\big(2^{-l}(\xi-\xi_0)\big)\varphi(2^{10}\xi)\big]\widehat{g}(t,\xi)\big|\,\diff \xi\\
	&\quad+t^{-1}\int_{-\infty}^\infty \big|(\partial_\xi\Phi)^{-1}\psi(\xi/\xi_0)\psi\big(2^{-l}(\xi-\xi_0)\big)\varphi(2^{10}\xi)\partial_\xi\widehat{g}(t,\xi)\big|\,\diff \xi\\
&=:A_{4l,1}+A_{4l,2}
	\end{aligned}
	\end{equation*}
with
\begin{equation*}
	\begin{aligned}
	A_{4l,1}&\lesssim 
	t^{-1}\int_{-\infty}^\infty \big|\partial_\xi\big[(\partial_\xi\Phi)^{-1}\psi(\xi/\xi_0)\psi\big(2^{-l}(\xi-\xi_0)\big)\big]\varphi(2^{10}\xi)\widehat{g}(t,\xi)\big|\,\diff \xi\\
	&\quad+t^{-1}\int_{-\infty}^\infty \big|(\partial_\xi\Phi)^{-1}\psi(\xi/\xi_0)\psi\big(2^{-l}(\xi-\xi_0)\big)\partial_\xi\varphi(2^{10}\xi)\widehat{g}(t,\xi)\big|\,\diff \xi\\
&=:A_{4l,1}^1+A_{4l,2}^2.
	\end{aligned}
	\end{equation*}
Observe that 
	$
	|\partial_\xi\Phi|
	\gtrsim  2^l\xi_0
	$
on the support of \(A_{4l}\). For the same reason as before, it suffices to focus on the new term \(A_{4l,2}^2\), which can be estimated as follows:
\begin{equation*}
\begin{aligned}
A_{4l,2}^2
\lesssim t^{-1}2^{-l}\xi_0^{-1}\|\widehat{g}\|_{L^\infty}\int_{-\infty}^\infty \psi(\xi/\xi_0)\psi\big(2^{-l}(\xi-\xi_0)\big)\,\diff \xi
\lesssim t^{-1}2^{-l},
\end{aligned}
\end{equation*}
which yields the desired bound \(t^{-1/2}\xi_0^{-1/2}\) by summation over \(l\geq l_0+1\) and using \(2^{l_0}\geq (t\xi_0)^{-1/2}\). Hence, we have also shown that \(A_4\) satisfies the desired bound \(t^{-1/2}\xi_0^{-1/2}\).

\end{proof}

\begin{proof}[Proof of \eqref{8}]  
The proof is similar to \cite{MR3961297} (see also \cite{MR4309883}). 	
Observe
	\begin{equation*}
	\begin{aligned}
	I_{H}:=\sum_{k\in \mathbb{Z}}\underbrace{\int_{-\infty}^\infty e^{\mathrm{i}t\Phi(t,\xi)}\widehat{P_kg}(t,\xi)\big(1-\varphi(2^{10}\xi)\big)\,\diff \xi}_{I_{H, k}}\approx \sum_{k\in \mathbb{N}}I_{H, k},
	\end{aligned}
	\end{equation*}
where one has used the property of the support of the integral. To show \eqref{8}, we first prove 	  
	\begin{equation}\label{14}
	\begin{aligned}
	|I_{H, k}|
	\lesssim t^{-1/2}2^{(1+s)k/2}\|\widehat{P_kg}\|_{L^\infty}
	+t^{-3/4}2^{(3s-1)k/4}(\|\widehat{P_kg}\|_{L^2}+2^k\|\partial\widehat{P_kg}\|_{L^2}),
	\end{aligned}
	\end{equation}
	and
	\begin{equation}\label{15}
	\begin{aligned}
	|I_{H, k}|
	\lesssim t^{-1/2}2^{(1+s)k/2}\|P_kg\|_{L^1}.
	\end{aligned}
	\end{equation}

We only show \eqref{14} since \eqref{15} is much easier.  	
When \(t\lesssim 2^{(s-1)k}\), it is easy to see that
	\begin{align*}
	|I_{H, k}|\lesssim 2^k\|\widehat{P_kg}\|_{L^\infty}\lesssim t^{-1/2}2^{(1+s)k/2}\|\widehat{P_kg}\|_{L^\infty}.
	\end{align*}
	It remains to consider \(t\gtrsim 2^{(s-1)k}\). Direct calculations yield
	\begin{align*}
	\bigg|\frac{\diff}{\diff \xi}\big((1+|\xi|^2)^{-s/2}\xi\big)\bigg|\geq c_0|\xi|^{-s},\quad \text{as}\
 |\xi|\geq 1/100
	\end{align*}
for some constant \(c_0>0\) independent of \(\xi\).		
	Let
	\begin{align*}
	\mathcal{I}:=\{k\in \mathbb{N}:  \frac{c_0}{4} |tx^{-1}|\leq 2^{sk}\leq 4c_0|tx^{-1}|\}.
	\end{align*}

	\noindent{\bf{Case 1: \(k\in \mathbb{N}\setminus\mathcal{I}\).}} 
	Integration by parts yields
	\begin{equation*}
	\begin{aligned}
	|I_{H, k}|
	&\lesssim t^{-1}\int_{-\infty}^\infty \big|\partial_\xi\big[(\partial_\xi\Phi)^{-1}\big(1-\varphi(2^{10}\xi)\big)\psi_k(\xi)\big]\widehat{P_kg}(t,\xi)\big|\,\diff \xi\\
	&\quad+t^{-1}\int_{-\infty}^\infty \big|(\partial_\xi\Phi)^{-1}\big(1-\varphi(2^{10}\xi)\big)\psi_k(\xi)\partial_\xi\widehat{P_kg}(t,\xi)\big|\,\diff \xi\\
&=:B_1+B_2.
	\end{aligned}
	\end{equation*}
	Notcing that 
	$
	|\partial_\xi\Phi(\xi)|\gtrsim \big||t^{-1}x|-c_0|\xi|^{-s}\big|\gtrsim 2^{-sk}
	$
on the support of \(I_{H, k}\),
	one may estimate 
	\begin{equation*}
	\begin{aligned}
	B_1
	\lesssim t^{-1}2^{(s-1/2)k}\|\widehat{P_kg}\|_{L^2},
	\end{aligned}
	\end{equation*}
	and
	\begin{equation*}
	\begin{aligned}
	B_2
	\lesssim t^{-1}2^{(s+1/2)k}\|\partial\widehat{P_kg}\|_{L^2}.
	\end{aligned}
	\end{equation*}
	Recalling \(t\gtrsim 2^{(s-1)k}\), we obtain
	\begin{equation*}
	\begin{aligned}
	|I_{H, k}|
	\lesssim t^{-3/4}2^{(3s-1)k/4}(\|\widehat{P_kg}\|_{L^2}+2^k\|\partial\widehat{P_kg}\|_{L^2}).
	\end{aligned}
	\end{equation*}

	\noindent{\bf{Case 2: \(k\in \mathcal{I}\).}} Notice that 
	\(\partial_\xi\Phi(\xi)=0\) on \(\mathbb{R}\setminus[-2^{-9},2^{-9}]\) has no root or two roots with opposite signs (corresponding to \(x>0\)). 
	We only consider the latter case since the other case is much easier. Denote by \(\xi_0\) the positive root of \(\partial_\xi\Phi(\xi)=0\) on \(\mathbb{R}\setminus[-2^{-9},2^{-9}]\).
	 Let \(l_0\) be the smallest integer satisfying  \(2^{l_0}\geq t^{-1/2}2^{(1+s)k/2}\).  Then, one has
	\begin{align*}
	I_{H, k}
	&= \int_{-\infty}^\infty e^{\mathrm{i}t\Phi(\xi)}\widehat{P_kg}(t,\xi)\big(1-\varphi(2^{10}\xi)\big)\varphi_{l_0}\big(\xi-\xi_0\big)\,\diff \xi\\
	&\quad+\sum_{l\geq l_0+1}\int_{-\infty}^\infty e^{\mathrm{i}t\Phi(\xi)}\widehat{P_kg}(t,\xi)\big(1-\varphi(2^{10}\xi)\big)\psi_l(\xi-\xi_0)\,\diff \xi\\
&=:J_{l_0}+\sum_{l\geq l_0+1}J_l.
	\end{align*}
First, one can easily check that
	\begin{align*}
	|J_{l_0}|
	\leq t^{-1/2}2^{(1+s)k/2}\|\widehat{P_kg}\|_{L^\infty}.
	\end{align*}
It remains to bound \(J_l\) for \(l\geq l_0+1\). Integration by parts yields
\begin{equation*}
\begin{aligned}
|J_l|
&\lesssim t^{-1}\int_{-\infty}^\infty \big|\partial_\xi\big[(\partial_\xi\Phi)^{-1}\big(1-\varphi(2^{10}\xi)\big)\psi_l(\xi-\xi_0)\psi_k(\xi)\big]\widehat{P_kg}(t,\xi)\big|\,\diff \xi\\
&\quad+t^{-1}\int_{-\infty}^\infty \big|(\partial_\xi\Phi)^{-1}\big(1-\varphi(2^{10}\xi)\big)\psi_l(\xi-\xi_0)\psi_k(\xi)\partial_\xi\widehat{P_kg}(t,\xi)\big|\,\diff \xi\\
&=:B_3+B_4.
\end{aligned}
\end{equation*}
Since  
$
|\partial_\xi\Phi(\xi)|\gtrsim \big||\xi_0|^{-s}-|\xi|^{-s}\big|\gtrsim 2^{l-(1+s)k}
$
on the support of \(J_l\), it follows that
\begin{equation*}
\begin{aligned}
B_3
\lesssim t^{-1}2^{-l+(1+s)k}\|\widehat{P_kg}\|_{L^\infty},
\end{aligned}
\end{equation*}
and
\begin{equation*}
\begin{aligned}
B_4
\lesssim t^{-1}2^{-\frac{l}{2}+(1+s)k}\|\partial\widehat{P_kg}\|_{L^2}.
\end{aligned}
\end{equation*}
Therefore
\begin{align*}
\sum_{l\geq l_0+1}J_l\lesssim t^{-1/2}2^{(1+s)k/2}\|\widehat{P_kg}\|_{L^\infty}
+t^{-3/4}2^{(3s-1)k/4}2^k\|\partial\widehat{P_kg}\|_{L^2}.
\end{align*}

Finally, we are in the position to show \eqref{8}. Let \(\theta=(N_0-s)/2>0\).
 If \(2^k\leq t^{1/(100N_0)}\), then we use \eqref{14} to deduce that 
\begin{equation}\label{16}
\begin{aligned}
2^{\theta k}|I_{H, k}|
\lesssim t^{-1/2}2^{(1+s)k/2}2^{\theta k}+t^{-3/4}2^{(3s-1)k/4}2^{(1+\theta)k}t^{1/6}\lesssim t^{-1/3}.
\end{aligned}
\end{equation}
If \(2^k\geq t^{1/(100N_0)}\), then it follows from \eqref{15} that 
\begin{equation}\label{17}
\begin{aligned}
2^{\theta k}|I_{H, k}|
&\lesssim t^{-1/2}2^{(1+s)k/2}2^{\theta k}2^{-k/2}2^{-N_0k/2}\\
&\quad\times\|P_k g\|_{H^{N_0}}^{1/2}(\|\widehat{P_k g}\|_{L^2}+2^k\|\partial\widehat{P_k g}\|_{L^2})^{1/2}\\
&\lesssim t^{-1/2}t^{1/12} t^{1/12}\lesssim t^{-1/3},
\end{aligned}
\end{equation}
where one has used (see for instance \cite{MR3961297})
\begin{equation*}
\begin{aligned}
\|P_k g\|_{L^2}\lesssim 2^{-k/2}\|P_k g\|_{L^2}^{1/2}(\|\widehat{P_k g}\|_{L^2}+2^k\|\partial\widehat{P_k g}\|_{L^2})^{1/2}.
\end{aligned}
\end{equation*}

Thanks to the decay factor \(2^{\theta k}\), \eqref{16} and \eqref{17} suffice to fulfil \eqref{8}.

\end{proof}

	\begin{remark}
	The linear estimates in \cite{MR1009061, MR840595} are used to prove the global existence and decay of small solutions to the generalized BBM equation:
\begin{equation*}
\begin{aligned}
	u_t+u_x+u^pu_x-u_{xxt}=0,
\end{aligned}
\end{equation*}		
when $p\geq 4$.
	
	On the other hand, Kwak and Munoz \cite{MR3931840} proved decay properties for small solutions of the generalized BBM equations, for any $p\geq 1$, in the region:
	\begin{equation*}
\begin{aligned}
	I(t)=(-\infty,-at)\cup ((1+b)t,\infty),\quad t>0,
\end{aligned}
\end{equation*}			
for any $b>0,\ a>1/8$.
	
	It would be interesting to extend this result to the generalized Fornberg-Whitham equations.
	\end{remark}

	\subsection{Solitary wave solutions}
	In order to study their long wave limit, we rescale \eqref{eq:1.1} with $p=1$ as
	\begin{equation}\label{resc}
\begin{aligned}
	u_t+\varepsilon uu_x-\mathcal K^\varepsilon_s u_x=0,
\end{aligned}
	\end{equation}
where $\varepsilon\ll 1$ and $\mathcal K_s^\varepsilon= (1-\varepsilon \partial_x^2)^{-s/2}$.
Observing that 
\begin{equation*}
\begin{aligned}
(1+\varepsilon\xi^2)^{-s/2}=1-\varepsilon s\frac{\xi^2}{2}+O(\varepsilon^2),
\end{aligned}
\end{equation*}
one gets from \eqref{resc} formally that
\begin{equation}\label{rescbis}
\begin{aligned}
	u_t+u_x+\varepsilon uu_x+\frac{\varepsilon s}{2} u_{xxx}=O(\varepsilon^2),
\end{aligned}
\end{equation}
suggesting that one obtains the KdV equation in the long wave limit. A similar fact has been used in \cite{MR2979975} for a class of nonlocal equations involving the Whitham equation to prove the existence of solitary wave solutions of those equations. In fact, one can check that the symbol of the linear part and nonlinearity of \eqref{eq:1.1} satisfy the {\bf Assumptions (A1)-(A3)} of \cite{MR2979975}, so that the existence result \cite[Theorem 1.2]{MR2979975} there applies in our case yielding existence of solitary wave solutions $u(x-\nu t)$ of the generalized Fornberg-Whitham equation for any $s>0$.
	 
	
	\subsection{Global weak entropy solutions}
	
We first give the definition of weak entropy solution of \eqref{eq:1.1} for \(p=1\):	
\begin{definition}[\cite{MR4270781}]
	Let $u_0\in L^1(\R)\cap L^\infty(\R)$. A function $u\in C([0,\infty),L^1(\R))$ that is bounded on $\R\times [0,T]$ for every $T>0$ is called a {\it weak entropy solution} of \eqref{eq:1.1}, if 
	\begin{equation*}
\begin{aligned}
	 &\int_0^\infty\int_{-\infty}^\infty [(|u(x,t)-\lambda|\partial_t\phi(x,t)+\frac{1}{2}\mathrm{sgn}(u(x,t)-\lambda)(u^2(x,t)-\lambda^2)\partial_x\phi(x,t)\\
	&\qquad\qquad-\mathrm{sgn}(u(x,t)-\lambda)\mathcal K_s'\star u(\cdot,t)(x)\phi(x,t)]\, \diff x\diff t\geq 0
\end{aligned}
	\end{equation*}
holds for arbitrary $\lambda\in \R$ and nonnegative test functions $\phi\in \mathcal D (\R\times (0,\infty))$.
	\end{definition}
	
When $s=2$, the weak entropy solution of the Fornberg-Whitham equation \eqref{FWter} was obtained in \cite{MR4270781}:
\begin{theorem}\label{Th2}
Let $u_0\in L^1(\R)\cap L^\infty(\R)$. Then the Cauchy problem to \eqref{FWter} with the initial data \(u(x,0)=u_0(x)\)
 has a unique entropy solution $u$. For any $t>0, x, y\in \R, x<y$, \(u\) satisfies the Oleinik type inequality
\begin{equation*}
\begin{aligned}
	u(y,t)-u(x,t)\leq \bigg(\frac{1}{t}+2+2t(1+2e^t\|u_0\|_{L^1}\bigg)(y-x).
\end{aligned}
\end{equation*}
	
Moreover, the following $L^1$ stability holds: if $v$ is the weak entropy solution corresponding to the initial data $v_0\in L^1(\R)\cap L^\infty(\R)$, then 
\begin{equation*}
\begin{aligned}
	\|u(t)-v(t)\|_{L^1}\leq e^t\|u_0-v_0\|_{L^1},\quad \forall t>0.
\end{aligned}
\end{equation*}
		
\end{theorem}

It would be interesting to extend this result to the generalized Fornberg-Whitham equation \eqref{eq:1.1} for all \(s>0\). 

One can refer to \cite{MR4571864, MR2084199, MR3527628} for other related works.

	\section{Final comments}
	
	
	
	We have shown that the Burgers equation with a dispersive perturbation based on Bessel potentials has a rather rich dynamics. It has a hyperbolic character (possibility of shock formation, existence of global weak solutions), and it displays dispersive properties: linear dispersive estimates and, in the long wave limit, a link with the KdV equation leading to  the existence of solitary waves.
	
	It has been shown in \cite{MR4201835, MR4309883} that the modified (cubic) fKdV equation with $-1<\alpha<0$ has also dispersive properties in the sense that it possesses global small  smooth solutions. It would be interesting to check if  this property still holds with a "Bessel type" dispersion and also for quadratic nonlinearities.




\end{document}